\documentclass{amsart}

\usepackage{amssymb,color}
\usepackage{amsfonts}
\usepackage{amsmath}
\usepackage{euscript}
\usepackage{enumerate}
\usepackage{graphics}

\newtheorem{theorem}{Theorem}[section]
\newtheorem{lemma}[theorem]{Lemma}
\newtheorem{note}[theorem]{Note}
\newtheorem{prop}[theorem]{Proposition}
\newtheorem{cor}[theorem]{Corollary}
\newtheorem{exa}[theorem]{Example}

\newtheorem{notation}[theorem]{Notation}

\newtheorem*{Theorem1'}{Theorem 1'}

\theoremstyle{definition}
\newtheorem{definition}[theorem]{Definition}

\theoremstyle{remark}

\numberwithin{equation}{section}

\newcommand \Q{{\mathbb Q}}
\newcommand \N{{\mathbb N}}

\newcommand \End{{\mathrm {End}}}
\newcommand \Hom{{\mathrm {Hom}}}
\newcommand \dm{{\mathrm {dim}}}
\newcommand \chr{{\mathrm {char}}}

\newcommand \GL{{\mathrm {GL}}}
\newcommand \FGL{{\mathrm {FGL}}}
\newcommand \la{{\lambda}}

\newcommand \al{{\alpha}}
\newcommand \be{{\beta}}
\newcommand \de{{\delta}}
\newcommand \ga{{\gamma}}
\newcommand \lm{{\lambda}}
\newcommand \si{{\sigma}}

\newcommand \ind{{\mathrm {ind}}}
\newcommand \res{{\mathrm {res}}}
\newcommand \Irr{{\mathrm {Irr}}}

\begin{document}

\title [\small{Representations of McLain groups}] {\small{Representations of McLain groups}}

\author{Fernando Szechtman}
\address{Department of Mathematics and Statistics, University of Regina}
\email{fernando.szechtman@gmail.com}
\thanks{The first and second authors were supported in part by an NSERC discovery grant}

\author{Allen Herman}
\address{Department of Mathematics and Statistics, University of Regina}
\email{Allen.Herman@uregina.ca}

\author{Mohammad A. Izadi}
\address{Department of Mathematics and Statistics, University of Regina}
\email{izadi.mohammadali9@gmail.com}





\begin{abstract} Basic modules of McLain groups $M=M(\Lambda,\leq, R)$
are defined and investigated. These are (possibly infinite
dimensional) analogues of Andr\'{e}'s supercharacters of $U_n(q)$.
The ring $R$ need not be finite or commutative and the field
underlying our representations is essentially arbitrary: we deal
with all characteristics, prime or zero, on an equal basis. The
set $\Lambda$, totally ordered by~$\leq$, is allowed to be
infinite. We show that distinct basic modules are disjoint,
determine the dimension of the endomorphism algebra of a basic
module, find when a basic module is irreducible, and exhibit a
full decomposition of a basic module as direct sum of irreducible
submodules, including their multiplicities. Several examples of
this decomposition are presented, and a criterion for a basic
module to be multiplicity-free is given. In general, not every
irreducible module of a McLain group is a constituent of a basic
module.
\end{abstract}

\maketitle

\section{Introduction}

In 1954 McLain \cite{M} constructed a family of groups that has
been a rich source of examples in group theory ever since (see
\cite{M2}, \cite{R},  \cite{HH}, \cite{Ro}, \cite{W}, \cite{DG},
\cite{CS}, \cite{Sz2}, for instance). A general McLain group $M=M(\Lambda,\leq,R)$ depends on a set
$\Lambda$, partially ordered by $\leq$, and an arbitrary ring $R$
with $1\neq 0$. Even though a partial order will do for some of our
purposes, for best results a total order will be required. In the
special case when $|\Lambda|=n$ is finite and $\leq$ is a total
order, $M=U_n(R)$ is the subgroup of $\GL_n(R)$ of all upper
triangular matrices with 1's on the main diagonal.

The main goal of this paper is define and study basic modules of
$M$, which are a generalization of the supercharacters of
$U_n(q)$, where $R=F_q$ is a finite field of characteristic~$p$.
We stress the fact that $\Lambda$ as well as $R$ are allowed to be
infinite, and $M$-modules are allowed to be infinite dimensional
over an arbitrary field $F$ (which need not have
characteristic 0). Moreover, the commutativity or not of~$R$ plays
no role whatsoever, so we will allow $R$ to be non-commutative. It
is perhaps surprising how of much of the theory of supercharacters
goes through in this context. A detailed description appears
below, after an overview of prior work on the subject. 

The representation theory of $U_n(q)$ draws considerable attention
due to its attractive nature and open problems. The literature on
the subject, as well as on the related algebra groups and Sylow
$p$-subgroups of classical groups, is too vast to review in full
detail and we will restrict ourselves to a limited overview.

One line of investigation was concerned with the degrees of the
complex irreducible characters of $U_n(q)$. In 1974 Lehrer
\cite{L} considered the so called elementary characters of
$U_n(q)$ as well as certain products of them, obtaining
(\cite[Corollary $5.2^\prime$]{L}) irreducible characters of
$U_n(q)$ of degree $q^c$ for every integer $c$ such that $0\leq
c\leq \mu(n)=(n-2)+(n-4)+\cdots$. Two decades later, Isaacs
\cite{I} confirmed that every irreducible character of, not only
$U_n(q)$, but also every $F_q$-algebra group, has $q$-power
degree. Isaacs' paper left open the question as to whether an
earlier assertion by Gutkin was true: is every irreducible
character of an algebra group induced from a linear character of
an algebra subgroup? This was partially confirmed by Andr\'{e}
\cite{A4} and fully by Halasi \cite{H}. An extension of Halasi's
result has recently been obtained by Boyarchenko \cite{B}. Isaacs'
result on character degrees does not extend directly to Sylow
$p$-subgroups of other classical groups, a fact recognized by
Isaacs himself and further confirmed by Gow, Marjoram and
Previtali \cite{GMP}. However, for $p$ odd,  Isaacs' result was
successfully extended to Sylow $p$-subgroups of symplectic,
orthogonal and unitary groups by Szegedy \cite{Sze}.

It is worth noting that all of Lehrer's elementary characters can
be constructed by lifting those associated to the position $(1,n)$
and that, when $q$ is odd, a character attached to this position is
nothing but a Weil character of $H\rtimes U_{n-2}(q)$, where $H$
is the Heisenberg group associated to a symplectic space of
dimension $2(n-2)$ over~$F_q$ and $U_{n-2}(q)$ is viewed as a
subgroup of a Sylow $p$-subgroup of the corresponding symplectic
group. Moreover, this identification remains valid when $F_q$  is
replaced by more general finite, commutative, local rings of odd characteristic
(see \cite{CMS} for the Weil representation in this context).

Another source of research is related to a conjecture by Higman
\cite{Hi}, to the effect that the number of irreducible characters
of $U_n(q)$ is an integer polynomial in~$q$. A sharpening of
Higman's conjecture, attributed by Lehrer \cite{L} to J.G.
Thompson, states that, for $0\leq c\leq \mu(n)$, the number of
irreducible characters of $U_n(q)$ of degree $q^c$ is an integer
polynomial in $q$. A further sharpening of the latter was
conjectured by Isaacs \cite{I2}, with evidence that the number of
irreducible characters of $U_n(q)$ of degree $q^c$ is a polynomial
in $q-1$ with non-negative integer coefficients. Work on these
conjectures has been intense. We refer the reader to \cite{G},
\cite{Le}, \cite{Lo}, \cite{Me}, \cite{Me2}, \cite{Mm}, \cite{HP},
\cite{T}, \cite{VA} and \cite{VA2} for work in this direction.

Rather than the degree of the irreducible characters of $U_n(q)$
or the total number of them, our attention is more related to the
study of supercharacters of $U_n(q)$ and their irreducible
constituents. As mentioned above, Lehrer \cite{L} proved that
certain products of elementary characters of $U_n(q)$ remain
irreducible. This prompted Andr\'{e} \cite{A} to consider more
general products of elementary characters, then called basic
characters and now called supercharacters after the work of
Diaconis and Isaacs \cite{DI}, who axiomatized a theory of
supecharacters and applied it to algebra groups. The main result
of \cite{A} is that every irreducible character of $U_n(q)$ is a
constituent of one and only one supercharacter. In particular,
distinct supercharacters are orthogonal. A supercharacter need not
be irreducible and Andr\'{e} \cite{A} gives a formula for the
inner product of a supercharacter with itself. All of this was
done under the assumption $p\geq n$, a restriction that was latter
removed in~\cite{A2}. The problem of ``finding" the irreducible
constituents of a supercharacter, as well as their multiplicities,
was addressed by Andr\'{e} in \cite[Theorem 2]{A3}, who also gave
\cite[Theorem 4]{A3} necessary and sufficient conditions for a
supercharacter to be a multiple of an irreducible character.
Andr\'{e} and Nicol\'{a}s \cite{AN} produced a fairly wide
generalization of the theory by considering not just $U_n(q)$ or
even $F_q$-algebra groups but the adjoint group
$G(A)=\{1+a\,|\,a\in A\}$ of a finite nilpotent ring $A$. They
were able to extend to this context the aforementioned result of
Halasi. Moreover, they defined and studied supercharacters in this
context, obtaining a decomposition (cf. \cite[Theorem 4.2]{AN})
much like the one given in \cite[Theorem 2]{A3}, as well necessary
and sufficient conditions (cf. \cite[Theorem 6.1]{AN}) for a
supercharacter to be a multiple of an irreducible character. At
the end of \cite{AN} they specialize to $U_n(q)$, reprove the
formula for the inner product of a supercharacter with itself, and
derive an irreducibility criterion for supercharacters (cf.
\cite[Theorem 7.1]{AN}).


Before we state our main results, we must describe our overall
assumption on $R$ and $F$. Our only assumption is the existence of
a right primitive linear character $R^+\to F^*$,
that is, a group homomorphism $R^+\to F^*$ having no non-zero
right ideals in its kernel. When $R$ is finite this condition has
been studied in detail (see \cite{La}, \cite{CG}, \cite{Wo} and
\cite{Ho}): it is left and right symmetric, and equivalent to $R$
being a Frobenius ring. This symmetry was left as an open question
in \cite{CG}, although it had already been established in
\cite{La}. It is reproved in \cite{Wo} and \cite{Ho}. As a
byproduct of our study of elementary modules, \S\ref{elemeM}
furnishes an independent proof of this symmetry by means of fully
ramified characters. The condition that $R$ have a primitive
linear character when $R$ is a finite, local and commutative ring
has already appeared in representation theory, e.g., in the
context of the Weil representation of symplectic groups over rings
(see \cite{CMS2} and \cite{Sz3}).

Let $K$ be a non-archimidean local field with ring of
integers~${\mathcal O}$, maximal ideal ${\mathfrak p}$ and residue
field $F_q={\mathcal O}/{\mathfrak p}$ of prime characteristic
$p$. Then $R={\mathcal O}/{\mathfrak p}^m$ is a finite, principal,
local, commutative ring of size $q^m$ affording a primitive linear
character when $F$ has a root of unity of order $p^m$ if
$\chr(K)=0$ and $p$ if $\chr(K)=p$. This can arguably be
considered as the most important example. A non-commutative
analogue can be obtained from Hilbert's twist ${\mathcal O}=D[[x;
\sigma]]$, the ring of skew power series twisted by an
automorphism $\sigma$ of a division ring $D$. This is a local ring
with Jacobson radical ${\mathfrak p}=(x)$. Then $R={\mathcal
O}/{\mathfrak p}^m$ affords a primitive linear character if and
only if so does $D$. In this regard, if $D$ has prime
characteristic $p$ it suffices that $F$ have a root of unity of
order $p$, while if $\chr(D)=0$ it is enough that, for some prime
$p$, $F$ have roots of unity of order $p^\ell$ for all $\ell\geq
1$. While in the above examples $R$ is always a principal ring,
this need not be the case. See \cite{CG} for further examples. The
actual choice of $R$ has no effect whatsoever on our arguments.

We begin our paper with a reminder, in \S\ref{mc}, of the
construction and basic properties of McLain groups. Basic tools to
deal with finite and infinite dimensional modules on an equal
basis are found in \S\ref{clif}, which essentially reproduces some
of the results from \cite{Sz} on the Clifford theory of possibly
infinite dimensional modules. In \S\ref{elemeM} we define and
study elementary modules of certain subgroups of $M$, which are themselves McLain groups,
generalizing the elementary modules constructed by Lehrer.
Our elementary modules give rise to the simplest example of two curious phenomena: an
irreducible module having no irreducible submodules (and hence failing
to be completely reducible) when restricted to a normal subgroup (further examples can
be found \cite{Sz}), and a family of (infinitely many) commuting
diagonalizable operators (acting on an infinite dimensional vector
space) having no common eigenvector.

Our paper properly begins after the above preliminary sections. From
\S\ref{mains} onwards we must assume that $\leq$ is a total order.
In particular, this allows us to extend to all of $M$ the action on elementary modules.
Let
$$
\Phi=\{(\al,\be)\in\Lambda\times\Lambda\,|\, \al<\be\}.
$$
For $\al<\be$ in $\Lambda$ we will abuse notation and denote by
$(\al,\be)$ not just the pair in~$\Phi$, but also the open
interval $\{\ga\in\Lambda\,|\, \al<\ga<\be\}$ in $\Lambda$. To
every triple $(\al,\be,\lm)$, where $(\al,\be)\in\Phi$ and
$\lm:R^+\to F^*$ is a right primitive linear character, there
corresponds an elementary $M$-module $V(\al,\be,\lm)$. This is
finite dimensional over $F$ if and only if the open interval
$(\al,\be)$ is empty, in which case $\dm(V(\al,\be,\lm))=1$, or
$R$ and $(\al,\be)$ are both finite, in which case
$\dm(V(\al,\be,\lm))=|R|^{|(\al,\be)|}$. Let
$$D=\{(\al_1,\be_1),\dots,(\al_m,\be_m)\}$$ be a basic subset of
$\Phi$, in the sense that all $\al_i$ (resp. all $\be_i$) are
distinct. Let $f$ be a choice function, from $D$ into the set of
all right primitive linear characters $R^+\to F^*$, say
$f(\al_i,\be_i)=\la_i$, $1\leq i\leq m$. Then the basic $M$-module
$V(D,f)$  is
$$V(D,f)=V(\al_1,\be_1,\lm_1)\otimes\cdots\otimes V(\al_m,\be_m,\lm_m).$$
This is in complete agreement with Andr\'{e}'s basic characters, as defined in \cite{A}.

Given a group $G$ and $G$-modules $X$ and $Y$, we define
$$
(X,Y)_G=\dm_F(\Hom_{FG}(X,Y)).
$$
Our first main result, Theorem \ref{disy}, states that basic modules are disjoint, that is,
\begin{equation}\label{nat}
(V(D,f),V(D',f'))_M=0\text{ if } (D,f)\neq (D',f').
\end{equation}
In words, distinct basic modules can only be connected via a zero
homomorphism.  We stress once more that this result as well as all
results stated below are, until further notice, valid in utter
generality: $R$ and $\Lambda$ are allowed to be infinite and a
basic $M$-module may be infinite dimensional. As is often the
case, a proof given in a more general context is conceptually
simpler (as ideas are stripped to the bone) and we believe this is
the case with Theorem \ref{disy} as well as others in this paper.

Let
$$
\Omega=\{(\al,\ga)\in\Phi\,|\,\exists\, (\al,\be),(\ga,\de)\in D\text{ such that
} \al<\ga<\be<\de\}.
$$
Then Theorem \ref{nest} gives the following irreducibility
criterion for basic modules: $V(D,f)$ is irreducible if and only
if $\Omega=\emptyset$. This is in complete agreement with the
corresponding result for $U_n(q)$ (cf. \cite[Theorem 7.1]{AN}).
Our proof of this irreducibility criterion is a direct application
of a well-known theorem of Gallagher \cite[Theorem 2]{Ga}. No
calculations of any kind are required. However, for full
generality, we must resort to the extension of Gallagher's theorem
found in \cite[Theorem 3.11]{Sz} (this extension is not required
for those interested only in $U_n(q)$, or even $U_n(R)$ with $R$
finite).

As in the case of $U_n(q)$, associated to each closed subset $\Gamma$ of $\Phi$ there is a
corresponding pattern subgroup $M(\Gamma)$ of $M$. Now $D$ gives rise to the closed subset
$$
\Gamma=\underset{1\leq i\leq
m}\bigcap(\Phi\setminus[\al_i,\rightarrow,\be_i))=\Phi\setminus(\underset{1\leq
i\leq m}\bigcup [\al_i,\rightarrow,\be_i)),
$$
where for $(\al,\be)\in\Phi$ we define
$$
[\al,\rightarrow,\be)=\{(\al,\ga)\in\Phi\,|\, \gamma<\be\}.
$$
If we let $H=M(\Gamma)$ then, much as in the case of
$U_n(q)$, there is a linear character $\lm:H\to F^*$,
corresponding to $(\lm_1,\dots,\lm_m)$, and a 1-dimensional
$H$-module $W$ upon which $H$ acts via $\lm$, such that (cf.
Theorem \ref{monom}):
$$
V(D,f)\cong \ind_H^M W.
$$
Again, this is in line with the corresponding result for $U_n(q)$, as found in \cite{A}. Let
$$
\Gamma_1=\Gamma\cup\Omega\text{ (disjoint union)},
$$
which is a closed subset of $\Phi$, and let $I=M(\Gamma_1)$. Complementing (\ref{nat}) we have
the following result (cf. Theorem \ref{inv}):
\begin{equation}\label{nat2}
(V(D,f),V(D,f))_M=[I:H].
\end{equation}
This is the perfect analogue for Andr\'{e}
formula (cf. \cite[Theorem 2]{A2}) for the inner product of a basic character with itself
(see also \cite[Corollary 2.10]{Le2}). Thus, $V(D,f)$ is irreducible if and
only if $I=H$ (cf. Theorem~\ref{coni}).


The most delicate part of our paper deals with the actual
decomposition of a basic character. This requires considerably
more effort than above. The most significant obstacle we face is
that a $G$-module $X$ may satisfy $(X,X)_G=1$ without being irreducible.
Obviously, irreducibility in this setting amounts to complete
reducibility. When $R$ and $\Lambda$ are finite this is a
non-issue, since the existence of a primitive linear character
implies $\chr(F)\nmid |R|$ (cf. Lemma \ref{cara}), and complete
reducibility follows from Maschke's theorem. By keeping $R$ finite
but allowing $\Lambda$ to be arbitrary, we were able to overcome
this obstacle by using the concept of ascendant subgroup. This is the same tool
successfully used by Meierfrankenfeld \cite{Me} and Wehrfritz
\cite{We}: if $G$ is an irreducible subgroup of $\FGL(V)$ (the
full finitary general linear group) and $H$ is an ascendant
subgroup of $G$ then $H$ is completely reducible (in this regard,
recall the aforementioned failure of a normal subgroup of an
irreducible group to be completely reducibly, encountered in
\S\ref{elemeM}). We actually require the opposite direction, from
$H$ to $G$, and have no need to resort to $\FGL(V)$. Our results
in this context are fairly general criteria for complete
reducibility, found in Theorem \ref{fund2} and its consequence,
Theorem \ref{mage}. The latter is a direct generalization of the
well-known irreducibility criterion of Mackey \cite{Ma},
originally proved in the context of finite groups and finite
dimensional modules over an algebraically closed field.

In order to be able to use the above general tools to decompose
basic modules of McLain groups, it is necessary to verify the
hypothesis these tools require. Much of this verification is
carried out in \S\ref{prema} as well as in Theorem \ref{xasc}.

The above verification allows us to reach our main result, namely
Theorem \ref{nuc}, which gives a full decomposition of a basic
module of a McLain group into irreducible constituents, including
multiplicities. Theorem \ref{nuc} is a complete generalization of
\cite[Theorem 2]{A3}. As in the case of $U_n(q)$, the decomposition of $V(D,f)$
is entirely controlled by that of $\ind_H^I W$. Explicitly, our main result
reads as follows.

\begin{theorem}\label{nuc23} Suppose $R$ is finite and $F$ is a splitting field for $I$ over
$\lm$ (in the sense of Definition \ref{skipi}). Then

(a) $\ind_H^I W$ is a completely reducible $I$-module of finite length $\leq [I:H]$ and
$$
\ind_H^I W=m_1V_1\oplus\cdots\oplus m_t V_t,
$$
where $\{V_1,\dots,V_t\}$ is a full set of representatives for the
isomorphism classes of irreducible $I$-modules lying over $W$ (or, equivalently, $\lm$). Moreover,
$$
(W,\res_H^I V_i)_H=\dm(V_i)=m_i,\quad \quad 1\leq i\leq t.
$$

(b) Assume, in addition, that $\Lambda$ is well-ordered by $\geq$, the inverse order of $\leq$
(this means: start with any well-order and impose its inverse on $\Lambda$). Then
$$
V(D,f)\cong m_1\ind_I^G V_1\oplus\cdots\oplus m_t \ind_I^G V_t,
$$
where $\ind_I^G V_1,\dots,\ind_I^G V_t$ are non-isomorphic irreducible $G$-modules, and
$$
(\ind_I^G V_i, V(D,f))_M=\dm(V_i)=m_i, \quad 1\leq i\leq t.
$$
\end{theorem}
The requirement that $\geq$ be a well-order is directly related to the complete reducibility issues
mentioned above. Specifically, Theorem \ref{xasc} shows that when $R$ is finite and $\geq$ is a well-order
then $I$ is a strongly ascendant subgroup of $M$ (in the sense of Definition \ref{asc}), which, in turn, ensures the complete reducibility of $V(D,f)$
by Theorem \ref{fund2}.

Some of the consequences of Theorem \ref{nuc23} appear to be
unknown even for $U_n(q)$. Indeed, as in the case of $U_n(q)$, we
have (cf. Theorem \ref{corin}) that $H$ is a normal subgroup of
$I$, and $I$ is included in the inertia group of $W$. What seems to
be unknown is when the action of $H$ on $W$ is extendible to $I$.
Theorem \ref{ext} answers this question: when $D$ has no special
triples, in the sense of Definition \ref{trip}. Combining Theorems
\ref{nuc} with Theorem \ref{ext} and the generalization of
Gallagher's theorem found in \cite[Theorem 3.11]{Sz}, we obtain a
much sharper decomposition of $V(D,f)$, as described in Theorem
\ref{tx}. This decomposition becomes even sharper if $I/H$ is
abelian, and this is stated in Corollary \ref{yeka}. Of course,
the action of $H$ on $W$ is, in general, not extendible to $I$,
and Example \ref{yeka2} illustrates how $V(D,f)$ decomposes in a
family of such cases. Combining all of the results of this
paragraph we obtain necessary and sufficient conditions for
$V(D,f)$ to be multiplicity-free, as described in Theorem
\ref{mulfe}. This result also seems to be unknown for $U_n(q)$.
Finally, Examples \ref{yeka2}, \ref{ulex} and \ref{coco}
illustrate, in our context, Andr\'{e}'s result on basic characters
which are multiplies of an irreducible character.



The main result of \cite{A} is actually false for $M$. Indeed, it
is shown in \cite{Sz2} that when $\Lambda=\Q$ is ordered as usual
and $R$ is a division ring whose characteristic is not at the same
time prime and equal to that of $F$, then $M$ has a faithful
irreducible module over $F$. But none of the basic modules for $M$
are faithful for such $\Lambda$, so not every irreducible $M$-module is
a constituent of a basic module. 

\section{McLain groups}\label{mc}

Let $R$ be a ring with $1\neq 0$, let $\Lambda$ be a non-empty set
partially ordered by $\le$, and set
$$
\Phi=\{(\al,\be)\in\Lambda\times\Lambda\,|\, \al<\be\}.
$$
Let $J$ be a free left $R$-module with basis $e_{\al\be}$,
$(\al,\be)\in\Phi$. We define a multiplication on $J$ by declaring
$$
e_{\al\be}e_{\be\gamma}=e_{\al\gamma},\quad
e_{\al\be}e_{\gamma\delta}=0\text{ if }\be\neq \gamma,
$$
and extending it to all of $J$ by $R$-bilinearity. This makes $J$
into a ring.

\begin{lemma} $J$ is a nil ring.
\end{lemma}

\begin{proof} This follows by induction on the length of the
longest chain in $\Phi$ present in an element of $J$ when written
as a linear combination of the $e_{\al\be}$.
\end{proof}

Adjoining an identity element to $J$, we obtain the group
$$
M=M(\Lambda)=\{1+x\,|\, x\in J\},
$$
the \emph{McLain group} associated to $(\Lambda,\le)$ over $R$. Every
$g\in M$ has the form
\begin{equation}
\label{geme}
g=1+\underset{(\al,\be)\in S(g)}\sum r_{\al\be}e_{\al\be},
\end{equation}
for a unique finite subset $S(g)$ of $\Phi$ and unique non-zero $r_{\alpha \beta}\in R$. Moreover,
every such an element is in $M$. The following commutator formula is valid in $M$:
$$
[1+re_{\al\be},1+se_{\be\gamma}]=1+rse_{\al\gamma},\quad
[1+re_{\al\be},1+se_{\gamma\delta}]=1\text{ if
}\be\neq\gamma\text{ and }\al\neq\delta.
$$
For $(\al,\be)\in\Phi$, consider the subgroup $M_{\al\be}$ of $M$ defined by
$$
M_{\al\be}=\{1+r e_{\al\be}\,|\, r\in R\}.
$$
Clearly, the following map is a group isomorphism from $R^+$ onto
$M_{\al\be}$:
\begin{equation}
\label{rma} r\mapsto 1+re_{\al\be},\quad r\in R^+.
\end{equation}

For the remainder of this section we fix a subset $\Gamma$ of
$\Phi$ that is \emph{closed}, in sense that $(\al,\be),(\be,\ga)\in \Gamma$
implies $(\al,\ga)\in \Gamma$. The {\em pattern} subgroup
$M(\Gamma)$ of $M$ corresponding to $\Gamma$ is given by
$$
M(\Gamma)=\langle M_{\al\be}\,|\, (\al,\be)\in\Gamma\rangle.
$$

Suppose $\Gamma\subseteq\Gamma_1$, where $\Gamma_1$ is also a closed subset of $\Phi$.
We say that $\Gamma$ is \emph{normal} in $\Gamma_1$ if $(\al,\be)\in \Gamma, (\be,\ga)\in
\Gamma_1$ implies $(\al,\ga)\in\Gamma$, and $(\al,\be)\in \Gamma_1, (\be,\ga)\in
\Gamma$ implies $(\al,\ga)\in\Gamma$. Note that $M(\Gamma)$ is a
normal subgroup of $M(\Gamma_1)$ if and only if $\Gamma$ is a normal subset of $\Gamma_1$.

We refer to $\Gamma$ as \emph{abelian} if $\Gamma$ contains no chains. Clearly,
$M(\Gamma)$ is an abelian subgroup of $M$ if and only if $\Gamma$ is an abelian subset of $\Phi$.

Given $(\al,\be)\in\Phi$, we consider
$$
[\al,\be]=\{(\gamma,\delta)\in\Phi\,|\, \al\leq \gamma<\delta\leq
\be\},
$$
$$
(\al,\leftarrow)=\{(\gamma,\delta)\in\Phi\,|\, \gamma<\al\},\quad
(\al,\rightarrow)=\{(\gamma,\delta)\in\Phi\,|\, \al<\gamma\},
$$
$$
(\rightarrow,\be)=\{(\gamma,\delta)\in\Phi\,|\, \be< \delta\},
\quad (\leftarrow,\be)=\{(\gamma,\delta)\in\Phi\,|\, \delta<\be
\},
$$
$$
(\downarrow,\be]=\{(\gamma,\delta)\in\Phi\,|\, \delta=\be\},\quad
[\al,\Rightarrow)=\{(\gamma,\delta)\in\Phi\,|\,\gamma=\alpha\},
$$
$$
[\al,\downarrow,\be]=[\al,\be]\cap (\downarrow,\be],
$$
$$
[\al,\rightarrow,\be]=[\al,\be]\cap [\al,\Rightarrow),
$$
as well as obvious variants of these obtained by interchanging
closed and open brackets, in which case
$\leq$ and $<$ are to be interchanged.

At different stages of the paper we will use the basic but
critical fact that $M(\Gamma)$ is actually a McLain group when
$\Gamma$ is any of the subsets $(\rightarrow,\be],
(\rightarrow,\be), (\al,\be)$.

All subsets of $\Gamma$ displayed above are closed. Moreover,
$(\al,\leftarrow)$ and $(\rightarrow,\be)$ are normal in $\Phi$, while
$[\al,\downarrow,\be]$ and $[\al,\rightarrow,\be]$ are abelian as
well as normal in $[\al,\be]$.

For the remainder of this section we suppose that $\leq$ is a total order on $\Lambda$.

\begin{definition}\label{tota} Consider the total order $\preceq$ on $\Phi$ given by
\begin{equation}\label{preo}
(\alpha,\beta) \preceq (\gamma,\delta) \iff \beta < \delta, \mbox{ or } \beta=\delta \mbox{ and } \alpha \le \gamma.
\end{equation}
\end{definition}

The elements of $M(\Gamma)$ can be uniquely expressed relative to this total order.

\begin{prop}\label{co1} Let $g\in M(\Gamma)$ be as in (\ref{geme}). Then
\begin{equation}\label{uni2}
g=\underset{(\al,\be)\in S(g)}\prod (1+r_{\al\be}e_{\al\be}),
\end{equation}
where the product is taken (from left to right) in decreasing $\preceq$-order.
\end{prop}

\begin{proof} This is trivial for $g=1$. Assume $g\neq 1$ and suppose the distinct $\beta$'s that occur for $(\alpha,\beta) \in S(g)$ are $\beta_1 < \beta_2 < \dots < \beta_n$. Then
$$g= (1 + \sum_{\alpha} r_{\alpha \beta_n} e_{\alpha \beta_n}) (1 + \sum_{\beta<\beta_n} r_{\alpha \beta} e_{\alpha \beta}).$$

By induction on the size of $S(g)$, the last factor can be expressed as desired.  Let $\alpha_1 < \alpha_2 < \dots < \alpha_m$ be the distinct $\alpha$'s for which $(\alpha,\beta_n) \in S(g)$.  Then
$$ (1 + \sum_{\alpha} r_{\alpha \beta_n} e_{\alpha \beta_n}) = (1 + r_{\alpha_m \beta_n}e_{\alpha_m \beta_n})(1+r_{\alpha_{(m-1)} \beta_n}e_{\alpha_{(m-1)} \beta_n}) \cdot \cdot \cdot (1 + r_{\alpha_1 \beta_n}e_{\alpha_1 \beta_n}). $$
\end{proof}

\begin{note}{\rm Since (\ref{geme}) is uniquely determined by $g$ then so is (\ref{uni2}).}
\end{note}

\begin{prop}\label{trax} Suppose $\Gamma\subseteq\Gamma_1$, where $\Gamma_1$ is a closed subset of $\Phi$
and $\Gamma$ is normal in $\Gamma_1$. Let $\Omega=\Gamma_1\setminus\Gamma$ and set
$$
T=\{1+x\,|\, x\in\mathrm{span}\,(e_{\al\be})_{(\al,\be)\in \Omega}\}.
$$
Then $T$ is transversal for $M(\Gamma)$ in $M(\Gamma_1)$.
\end{prop}

\begin{proof} By Proposition \ref{co1}, $T$ consists of all $g\in M(\Gamma_1)$ as in (\ref{uni2}) with $S(g)\subseteq \Omega$.
Let $y\in M(\Gamma_1)$. Since $M(\Gamma)\unlhd M(\Gamma_1)$,
Proposition \ref{co1} ensures the existence of $t\in T$ and $h\in M(\Gamma)$ such
that $y=th$. Suppose $s,t\in T$ and $sM(\Gamma)=tM(\Gamma)$.
We claim that $s=t$. Indeed, let
\begin{equation}\label{lefmo}
s=(1+ae_{\al\be})u\text{ and }t=(1+be_{\al\be})v
\end{equation}
be the canonical expressions of $s$ and $t$ as elements of $M(\Gamma_1)$ ensured by Proposition~\ref{co1}, where $(\al,\be)\in\Omega$
and $u,v$ involve $(\ga,\de)\in\Omega$ such that $(\ga,\de)\prec (\al,\be)$ as defined by (\ref{preo}). Possible trivial
factors have been allowed so as to treat $s$ and $t$ simultaneously. Then
\begin{equation}\label{cuad}
st^{-1}=(1+ae_{\al\be})(1-be_{\al\be})w,
\end{equation}
where
$$
w=(1+be_{\al\be})uv^{-1}(1+be_{\al\be})^{-1}.
$$
Here $uv^{-1}$ is a product of  factors $1+ce_{\ga\de}$ with $(\ga,\de)\prec (\al,\be)$. In particular, $\ga<\be$, so either
$1+be_{\al\be}$ and $1+ce_{\ga\de}$ commute or $\ga<\de=\al$, in which case
$$
(1+be_{\al\be})(1+ce_{\ga\al})(1+be_{\al\be})^{-1}=(1+ce_{\ga\al})(1-cbe_{\ga\be}).
$$
Thus $w=1+x$, where $x$ is an $R$-linear combination of $e_{\ga\de}$ with $(\ga,\de)\prec (\al,\be)$. Therefore by (\ref{cuad}),
$$
st^{-1}=(1+(a-b)e_{\al\be})(1+x)=1+(a-b)e_{\al\be}+x.
$$
On the other hand, since $M(\Gamma)s=M(\Gamma)t$ and $\Omega\cap\Gamma=\emptyset$, the coefficient of $e_{\al\be}$ when
$st^{-1}\in M(\Gamma)$ is written as in (\ref{geme}) must be 0. This shows $a=b$. Now cancel the leftmost factors of $s$ and $t$ in (\ref{lefmo}) and repeat.
\end{proof}

\begin{definition} Given a subset $\Delta$ of $\Phi$ we define $[\Delta,\Delta]$ to be the closed subset of $\Phi$
consisting of all $(\al,\be)$ such that there is a chain $(\ga_1,\ga_2),\dots,(\ga_{n-1},\ga_n)\in \Delta$, $n\geq 3$,
satisfying $\ga_1=\al$ and $\ga_n=\be$.
\end{definition}

\begin{prop}\label{co2} (a) $[M(\Gamma),M(\Gamma)]=M([\Gamma,\Gamma])$.

(b) Let $\rho:M(\Gamma)\to M(\Gamma)/[M(\Gamma),M(\Gamma)]$ be the canonical projection, and consider the map
$$\Theta: \underset{(\al,\be)\in\Gamma\setminus [\Gamma,\Gamma]}\prod
M_{\al\be}\to M(\Gamma)/[M(\Gamma),M(\Gamma)],$$
where the left hand side is the external direct product of the given $M_{\al\be}$, and $\Theta$ is defined by
$$
\underset{(\al,\be)\in\Gamma\setminus[\Gamma,\Gamma]}\prod (1+r_{\al\be}e_{\al\be})\to \rho(\underset{(\al,\be)\in\Gamma\setminus[\Gamma,\Gamma]}\prod (1+r_{\al\be}e_{\al\be})).
$$
Here almost all $r_{\al\be}$ are 0 and the product on the right hand side is taken in decreasing order under $\preceq$
(the order on the left hand side is obviously irrelevant). Then $\Theta$ is a group isomorphism.
\end{prop}

\begin{proof} (a) By definition, $M([\Gamma,\Gamma])$ is generated by all $M_{\sigma\tau}$ with $(\sigma,\tau)\in [\Gamma,\Gamma]$.
Given such $(\sigma,\tau)$ there are $(\al_1,\al_2),\dots,(\al_{n-1},\al_n)\in\Gamma$, $n\geq 3$, such that $\al_1=\sigma$ and~$\al_n=\tau$.
A repeated application of the commutator formula shows that $M_{\sigma\tau}$ is contained in $[M(\Gamma),M(\Gamma)]$, and therefore
$M([\Gamma,\Gamma])\subseteq [M(\Gamma),M(\Gamma)]$.

On the other hand, $M(\Gamma)$ is generated by all $M_{\al\be}$, $(\al,\be)\in \Gamma$.
Thus, by the commutator formula, $M(\Gamma)/M([\Gamma,\Gamma])$ is abelian, whence $[M(\Gamma),M(\Gamma)]\subseteq M([\Gamma,\Gamma])$.

(b) This follows from Proposition \ref{co1} as well as from Proposition \ref{trax}.
\end{proof}

\section{Clifford theory for infinite dimensional modules}\label{clif}

We fix a field $F$ for the remainder of the paper.
Let $N\unlhd G$ be groups and let $W$ be an $N$-module. For $g\in
G$, consider the $N$-module~$W^g$, whose underlying $F$-vector
space is $W$, acted upon by $N$ as follows:
$$
x\cdot w= (g x g^{-1})w,\quad  x\in N,w\in W.
$$
Then
$$
I_G(W)=\{g\in G\,|\, W^g\cong W\}
$$
is a subgroup of $G$ called the inertia group of $W$ (cf.
\cite[Lemma 3.1]{Sz}.

\begin{theorem}\label{cl2} Let $N\unlhd G$ be groups and
let $W$ be an irreducible $N$-module with inertia group $T$. Then
$S\mapsto\ind_T^G S$ yields a bijective correspondence between the
isomorphism classes of irreducible modules of $T$ and $G$ lying
over $W$.

In particular, if $I_G(W)=N$ then $V=\ind_N^G W$ is irreducible
and if, in addition, $\End_N(W)=F$ then $\End_G(V)=F$ as well.
\end{theorem}

\begin{proof} See \cite[Theorem 3.5]{Sz} for the first assertion. As for
the second, by Frobenius reciprocity (cf. \cite[Theorem 5.3]{Sz}),
$\End_N(W)=\Hom_N(W,V)\cong_F\End_G(V)$.
\end{proof}

\begin{lemma}\label{forma} Let $N\unlhd G$ be groups and let $W$ be an irreducible $G$-module
such that $\res_N^G W$ remains irreducible and $\End_N(W)=F$. Let
$U_1,U_2$ be $G$-modules acted upon trivially by $N$ and suppose
$T\in\Hom_G(U_1\otimes W,U_2\otimes W)$. Then $T=S\otimes 1$,
where $S\in\Hom_{G/N}(U_1,U_2)$.
\end{lemma}

\begin{proof} Since $T\in\Hom_N(U_1\otimes W,U_2\otimes W)$, \cite[Lemma 3.7]{Sz}
implies $T=S\otimes 1$, where $S\in\Hom_F(U_1,U_2)$. But
$T\in\Hom_G(U_1\otimes W,U_2\otimes W)$, so
$S\in\Hom_{G/N}(U_1,U_2)$.
\end{proof}

\begin{cor}\label{arreglo}  Let $N\unlhd G$ be groups and let $W$ be a $G$-module
such that $\res_N^G W$ is a multiplicity-free, completely
reducible $N$-module, such that  $\End_N(X)=F$ for every
irreducible constituent $X$ of $\res_N^G W$. Let $U_1,U_2$ be
$G$-modules acted upon trivially by $N$ and suppose
$T\in\Hom_G(U_1\otimes W,U_2\otimes W)$. Then $T=S\otimes 1$,
where $S\in\Hom_{G/N}(U_1,U_2)$.
\end{cor}

\begin{theorem}\label{cli4}  Let $N\unlhd G$ be groups and let $W$ be an irreducible
$G$-module with $\res_N^G W$  irreducible and $\End_N(W)=F$. Then $U\mapsto U\otimes W$
yields a bijective correspondence between the
isomorphism classes of irreducible $G$-modules acted upon
trivially by $N$ and irreducible $G$-modules lying over $W$.
\end{theorem}

\begin{proof} This can be found in \cite[Theorem 3.11]{Sz}.
\end{proof}

\begin{theorem}\label{cli5} Let $G$ be a group with subgroups $H_1$, $H_2$ and irreducible $G$-modules
$V_1$, $V_2$ such that: $H_1$ acts trivially on $V_2$; $H_2$ acts
trivially on $V_1$; $V_i$ is irreducible as $H_i$-module and
$\End_{H_i}(V_i)=F$ for $1\leq i\leq 2$. Then $V=V_1\otimes V_2$ is
an irreducible module for $H=\langle H_1,H_2\rangle$ and
$\End_H(V)=F$.
\end{theorem}

\begin{proof} See the proof of \cite[Theorem 3.10]{Sz} for the first
assertion. The second follows from \cite[Lemma 3.7]{Sz}.
\end{proof}

\begin{cor}\label{cli6}  Let $G$ be a group with subgroups $H_1,\dots, H_n$ and irreducible $G$-modules
$V_1,\dots,V_n$ such that: $H_i$ acts trivially on each $V_j$,
$j\neq i$; $V_i$ is irreducible as $H_i$-module and
$\End_{H_i}(V_i)=F$ for $1\leq i\leq n$. Then
$V=V_1\otimes\cdots\otimes V_n$ is an irreducible $G$-module.
\end{cor}

\begin{proof} This follows from Theorem \ref{cli5} by induction.
\end{proof}

\section{Elementary modules}\label{elemeM}

\begin{definition} A linear character $\lm:R^+\to F^*$ is said to be \emph{right} (resp. \emph{left})
\emph{primitive} if the only right (resp. left) ideal of $R$ contained in the kernel of $\lm$ is $(0)$.
\end{definition}

We assume for the remainder of the paper that $R$ admits a right primitive linear character $\lm: R^+\to
F^*$.

\begin{lemma}\label{cara} Suppose $F$ has prime characteristic $p$. Then $R^+$ has no element of order $p$.
\end{lemma}

\begin{proof} The set $\{r\in R\,|\, p\cdot r=0\}$ is an ideal of $R$ contained in the kernel of every group homomorphism $R^+\to F^*$.
\end{proof}

\begin{lemma}\label{cara2} Suppose $R^+$ has finite exponent. Then $F^*$ has a root of unity of order $\mathrm{exp}(R^+)$.
\end{lemma}

\begin{proof} The prime factorization $\mathrm{exp}(R^+)=p_1^{a_1}\cdots p_n^{a_n}$ yields the factorization
$R^+=R_1\oplus\cdots\oplus R_n$, as rings, where $R_i=\{r\in R\,|\, p_i^{a_i}\cdot r=0\}$. Now $(0)\neq p_i^{a_i-1}R_i\unlhd R$,
so for each $i$ there is $r_i\in R_i$ such that $\lm(p_i^{a_i-1}r_i)\neq 1$. But $\lm(p_i^{a_i} r_i)=1$, so each $\lm(r_i)$
has order $p_i^{a_i}$, whence $\lm(r_1)\cdots\lm(r_n)$ has order $\mathrm{exp}(R^+)$.
\end{proof}

Given  $(\al,\be)\in\Phi$, consider the normal subgroup
$M^{\al\be}$ of $M([\al,\be])$ defined by
$$
M^{\al\be}=M([\al,\downarrow,\be])M([\al,\rightarrow,\be]).
$$
Let $W=Fw$ be a 1-dimensional  $M_{\al\be}$-module upon which  $M_{\al\be}$
acts, through (\ref{rma}), by means of $\lm$.

Note that $M([\al,\downarrow,\be])$ is the direct product of
$M((\al,\downarrow,\be])$ and $M_{\al\be}$, so we can extend $\lm$
to a linear character $\lm'$ of $M([\al,\downarrow,\be])$ that is
trivial on $M((\al,\downarrow,\be])$. We make
$M([\al,\downarrow,\be])$ act on $W$ via $\lm'$.

Now $M([\al,\downarrow,\be])$ is normal in  $M^{\al\be}$, and the
inertia group of the $M([\al,\downarrow,\be])$-module $W$ in
$M^{\al\be}$ is $M([\al,\downarrow,\be])$. The latter holds
because $\lm$ is right primitive and $R$ has $1$. It follows from
Theorem \ref{cl2} that
$$
U(\al,\be,\lm)=\ind_{M([\al,\downarrow,\be])}^{M^{\al\be}} W
$$
is an irreducible $M^{\al\be}$-module satisfying
$\End_{M^{\al\be}}U(\al,\be,\la)=F$. Moreover, if
the open interval $(\al,\be)$ of $\Lambda$ is finite, it is easy to see that
$$\dm(U(\al,\be,\lm))=|R|^{|(\al,\be)|}.$$

On the other hand, $M([\al,\downarrow,\be])$ is normal in
$M([\al,\be])$ and the inertia group of the
$M([\al,\downarrow,\be])$-module $W$ in $M([\al,\be])$ is
$M([\al,\downarrow,\be])\rtimes M((\al,\be))$. As $W$ is
1-dimensional, the action of $M([\al,\downarrow,\be])$ on $W$ is
extendible to $M([\al,\downarrow,\be])\rtimes M((\al,\be))$ by
letting $M((\al,\be))$ act trivially on $W$. It follows from
Theorem \ref{cl2} that
$$
V(\al,\be,\lm)=\ind_{M([\al,\downarrow,\be])\rtimes M((\al,\be))}^{M([\al,\be])} W
$$
is an irreducible $M([\al,\be])$-module, known as \emph{elementary}. But
$$
M^{\al\be}=M([\al,\downarrow,\be])\rtimes M([\al,\rightarrow,\be))
$$
and
$$
M([\al,\be])=\big(M([\al,\downarrow,\be])\rtimes M([\al,\rightarrow,\be))\big)\rtimes M((\al,\be))
$$
so
$$
\res^{M([\al,\be])}_{M^{\al\be}} V(\al,\be,\lm)\cong
U(\al,\be,\lm).
$$
We have shown

\begin{theorem}\label{schro} Let $(\al,\be)\in \Phi$ and suppose $\la:R^+\to
F^*$ is primitive. Then $V(\al,\be,\la)$ is an irreducible
$M([\al,\be])$-module, which remains irreducible as an
$M^{\al\be}$-module and satisfies
$\End_{M^{\al\be}}V(\al,\be,\la)=F$.
\end{theorem}

Next we address the classification of irreducible $M^{\al\be}$-modules lying over $\lm$.

\begin{theorem}\label{solop} Suppose $R$ is finite and the interval $(\al,\be)$ of $\Lambda$ is
also finite. Then $U(\al,\be,\lm)$ is the only irreducible
$M^{\al\be}$-module, up to isomorphism, lying over $\lm$.
\end{theorem}

\begin{proof} By Lemma \ref{cara} and Maschke's theorem, every $M^{\al\be}$-module is completely reducible.
By construction, $W$ enters $|R|^{|(\al,\be)|}$ times in $U(\al,\be,\lm)$, so $U(\al,\be,\lm)$ enters
$|R|^{|(\al,\be)|}$ times in $\ind_{M_{\al\be}}^{M^{\al\be}} W$ by Frobenius reciprocity. But $|R|^{|(\al,\be)|}U(\al,\be,\lm)$
and $\ind_{M_{\al\be}}^{M^{\al\be}}W $ have the same degree, so the result follows.
\end{proof}

\begin{cor}\label{xos} Suppose $R$ is finite. Then $\lm$ is also left primitive.
\end{cor}

\begin{proof} Let $K$ be the largest left ideal of~$R$ contained in the kernel of
$\lm$. By Theorem~\ref{solop} there is one and only one irreducible module of
$M=U_3(R)$, up to isomorphism, lying over $\lm$, viewed as a
linear character of $M_{13}$. Let $S=M_{12}M_{13}\unlhd M$ and
extend $\lm$ to $\mu:S\to F^*$ so that $\mu$ is trivial on
$M_{12}$. The inertia group of $\mu$ is $S\rtimes T$, where $T$
consists of all $1+re_{23}$ with $r\in K$. Since $T$ stabilizes
$\mu$, we can extend $\mu$ to $\nu:S\rtimes T\to F^*$, so that
$\nu$ is trivial on $T$. Now $S\rtimes T$ is normal in $M$ and the
inertia group of $\nu$ is $S\rtimes T$. Induction produces an
irreducible $M$-module of degree $[R:K]$ lying over~$\lm$. By
uniqueness, $K=(0)$.
\end{proof}

Naturally, we could have started with the assumption that $\lm$ is
left primitive to deduce that $\lm$ is right primitive. We thus obtain, through representation theory, an answer to the question posed
by Claasen and Goldbach \cite[\S 8]{CG}, already answered by Wood
\cite[Theorem 4.3]{Wo} by a different method.

\begin{theorem}\label{pkl} Suppose the interval $(\al,\be)$ is infinite. Then there are
uncountably many non-isomorphic irreducible $M^{\al\be}$-modules
lying over $\lm$. All are completely reducible upon restriction to
$M([\al,\downarrow,\be])$, none is completely reducible upon
restriction to $M([\al,\rightarrow,\be])$, and therefore none has
irreducible submodules upon restriction to the normal subgroup
$M([\al,\rightarrow,\be])$.
\end{theorem}

\begin{proof} Given a family $(\mu_\ga)_{\ga\in (\al,\be)}$ of linear characters $R\to F^*$, we
extend $\lm$ to a linear character
$\mu:M([\al,\downarrow,\be]) \to F^*$ that agrees
with $\mu_\ga$ on $M_{\ga\be}$ for all $\ga$ in $(\al,\be)$. The
inertia group of $\mu$ is still $M([\al,\downarrow,\be])$.
Inducing up to $M^{\al\be}$ we obtain an irreducible
$M^{\al\be}$-module $U_\mu$ lying over~$\lm$.  A conjugate of
$\mu$ by an element of $M([\al,\rightarrow,\be])$ agrees with
$\mu$ on $M_{\ga\be}$ for all but finitely many $\ga\in
(\al,\be)$. On the other hand, by Clifford's theorem, if $U$ is an
irreducible $M^{\al\be}$-module, then $U\cong U_\mu$ if and only
if $U$ lies over a $M([\al,\rightarrow,\be])$-conjugate of $\mu$.

Consider the equivalence
relation on the set $X$ of infinite subsets of $\N$, given by $A\sim
B$ if $A\Delta B$ (symmetric difference of $A$ and $B$) is finite.
Each equivalence class has countably many elements and $X$ is uncountable,
so the number of equivalence classes is uncountable. Since $(\al,\be)$ is infinite,
we arrive at the same conclusion if we replace $\N$ by $(\al,\be)$. It
follows from above that, by varying $\mu$, we obtain uncountably many non-isomorphic irreducible
$M^{\al\be}$-modules that lie over $\lm$ and are completely reducible upon restriction to $M([\al,\downarrow,\be])$.

As $M_{\al\be}$ acts via scalar operators on $U_\mu$, our claim for $\res^{M^{\al\be}}_{M([\al,\rightarrow,\be])} U_\mu$ is equivalent to that for
$\res^{M^{\al\be}}_{M([\al,\rightarrow,\be))} U_\mu$. Restriction to $M([\al,\rightarrow,\be))$ yields the regular $M([\al,\rightarrow,\be))$-module, say $P$. 
The epimorphism $P\to F$ shows that a
supposed complement to the augmentation ideal must be trivial. But
$P$ has no trivial submodule because $(\al,\be)$, and hence $M([\al,\rightarrow,\be))$, is infinite. Thus
$\res^{M^{\al\be}}_{M([\al,\rightarrow,\be])} U_\mu$ is not completely reducible, and therefore lacks irreducible $M([\al,\rightarrow,\be])$-submodules, by Clifford's theorem.
\end{proof}

\begin{note}{\rm When $\mathrm{exp}(R^+)$ is finite, Theorem \ref{pkl} gives an example of an infinite family of diagonalizable linear operators acting on an infinite dimensional vector space that are are simultaneously diagonalizable and which, in fact, possess no common eigenvector.
If either the space or the span of the family are finite dimensional, this phenomenon is impossible. The family simply consists
of the operators corresponding to each $\ga\in [\al,\rightarrow,\be)$ acting on $U_\mu$. They commute because $[\al,\rightarrow,\be)$
is abelian and they are individually diagonalizable because they are all annihilated by the polynomial $t^{\mathrm{exp}(R^+)}-1\in F[t]$,
which splits completely over $F$ by Lemma~\ref{cara2}.}
\end{note}

Next we show that $M^{\al\be}$ has yet another uncountable family of
non-isomorphic irreducible modules, none of which is isomorphic to any of
the modules constructed in Theorem \ref{pkl}.

 \begin{theorem} Suppose $\lm$ is also left primitive
 (this is automatic if $R$ is finite, by Lemma \ref{xos}, and also if $R$ is commutative) and the interval $(\al,\be)$ is infinite. Then there are
uncountably many non-isomorphic irreducible $M^{\al\be}$-modules
lying over~$\lm$. All are completely reducible upon restriction to
$M([\al,\rightarrow,\be])$, none is completely reducible upon
restriction to $M([\al,\downarrow,\be])$, and therefore none has
irreducible submodules upon restriction to the normal subgroup
$M([\al,\downarrow,\be])$.
\end{theorem}

\begin{proof} Reason as above. First extend $\lm$
to $[\al,\rightarrow,\be]$ and then induce up to $M^{\al\be}$.
\end{proof}

\section{Basic modules}\label{mains}

We assume for the remainder of the paper that $\leq$ is a total
order on $\Lambda$. Then
\begin{equation}
\label{desz} M=M(\al,\leftarrow)M(\rightarrow,\be)\rtimes
M([\al,\be]),
\end{equation}
so that each elementary module $V(\al,\be,\lm)$ can be viewed as
an $M$-module acted upon trivially by
$M(\al,\leftarrow)M(\rightarrow,\be)$.

\begin{definition} A subset $D=\{(\al_1,\be_1),\dots,(\al_m,\be_m)\}$ of $\Phi$ is said to be \emph{basic}
if all $\al_i$ (resp. all $\be_i$) are distinct.

Given a basic subset $D$ of $\Phi$ and a function $f$ from $D$
into the set of all right primitive linear characters $R^+\to F^*$, say
$f(\al_i,\be_i)=\la_i$, $1\leq i\leq m$, the \emph{basic} $M$-module
$V(D,f)$ associated to $(D,f)$ is given by
$$V(D,f)=V(\al_m,\be_m,\lm_m)\otimes\cdots\otimes V(\al_1,\be_1,\lm_1).$$
\end{definition}

\begin{theorem}\label{disy} Suppose $U=V(D,f)$ and $U'=V(D',f')$ are basic
$M$-modules satisfying
$$
\Hom_M(U,U')\neq 0.
$$
Then $D=D'$ and $f=f'$.
\end{theorem}

\begin{proof} Relabelling, if necessary, we may assume that
$\be_1$ (resp. $\be_1'$) is the largest element of
$\{\be_1,\dots,\be_m\}$ (resp. $\{\be'_1,\dots,\be'_{m'}\}$).

By construction, each $V(\al_i,\be_i,\lm_i)$ (resp.
$V(\al'_i,\be'_i,\lm'_i)$), $i>1$, is acted upon trivially
by $M_{\al_1\be_1}$ (resp. $M_{\al'_1\be'_1}$). Therefore
$M_{\al_1\be_1}$ (resp. $M_{\al'_1\be'_1}$) acts on $U$ (resp.
$U'$) via scalar operators determined by $\lm_1$ (resp. $\lm'_1$).

Suppose, if possible, that $\be_1<\be'_1$ (resp. $\be_1'<\be_1$).
Then $M_{\al'_1\be'_1}$ (resp. $M_{\al_1\be_1}$) acts trivially on
$U$ (resp. $U'$). Thus
$$\Hom_{M_{\al'_1\be'_1}}(U,U')=0\text{ (resp. }\Hom_{M_{\al_1\be_1}}(U,U')=0),
$$
whence $\Hom_{M}(U,U')=0$, a
contradiction. This forces $\be_1=\be'_1$.

Suppose, if possible, that $\al_1'<\al_1$ (resp. $\al_1<\al'_1$).
Then $M_{\al'_1\be'_1}$ (resp. $M_{\al_1\be_1}$) acts trivially on
$U$ (resp. $U'$), regardless of how $\al_1$ and $\al'_1$ compare
to the other $\al_i$ and $\al'_i$. As above, this leads to the
contradiction $\Hom_{M}(U,U')=0$, so $\al_1=\al'_1$.

Since $M_{\al_1\be_1}$ acts on $U$ (resp. $U'$) via scalar
operators determined by $\lm_1$ (resp. $\lm'_1$), the condition
$\Hom_M(U,U')\neq 0$ forces $\lm_1=\lm'_1$.

On the other hand, we have the decomposition
$$
M=M((\rightarrow, \be_1))\rtimes M((\leftarrow,\be_1]),
$$
where $M((\rightarrow,\be_1))$ acts trivially on $U$ and $U'$, so
\begin{equation}
\label{noze} \Hom_{M((\leftarrow,\be_1])}(U,U')\neq 0.
\end{equation}
Set
$$
W=V(\al_m,\be_m,\lm_m)\otimes\cdots\otimes V(\al_2,\be_2,\lm_2),\,
W'=V(\al'_{m'},\be'_{m'},\lm'_{m'})\otimes\cdots\otimes
V(\al'_2,\be'_2,\lm'_2),
$$
with $W$ (resp. $W'$) the trivial module if $m=1$ (resp. $m'=1$).
We also have the decomposition
\begin{equation}
\label{noze2} M((\leftarrow,\be_1])=M((\downarrow,\be_1])\rtimes
M((\leftarrow,\be_1)),
\end{equation}
where $M((\downarrow,\be_1])$ is a normal subgroup of
$M((\leftarrow,\be_1])$ acting trivially on $W$ and~$W'$.
Moreover, by above,
$$
U=W\otimes V(\al_1,\be_1,\lm_1)\text{ and }U'=W'\otimes
V(\al_1,\be_1,\lm_1).
$$
We readily see that
$\res_{M((\downarrow,\be_1])}^{M((\leftarrow,\be_1])}
V(\al_1,\be_1,\lm_1)$ is a multiplicity-free completely reducible
$M((\downarrow,\be_1])$-module with 1-dimensional irreducible
constituents. It follows from Corollary \ref{arreglo},
(\ref{noze}) and  (\ref{noze2}) that
$$
\Hom_{M((\leftarrow,\be_1))}(W,W')\neq 0,
$$
where $M((\leftarrow,\be_1))$ is a McLain group. The above
argument shows that $m=1$ if and only if $m'=1$, while otherwise
$W$ and $W'$ are basic $M((\leftarrow,\be_1))$-modules, and
induction applies.
\end{proof}


\begin{notation}{\rm We fix a basic
subset $D=\{(\al_1,\be_1),\dots,(\al_m,\be_m)\}$ of $\Phi$ for the
remainder of the paper, as well as a function  $f$ from $D$ into
the set of all right primitive linear characters $R^+\to F^*$, say
$f(\al_i,\be_i)=\la_i$, $1\leq i\leq m$. Moreover, in the presence
of $(D,f)$, the following notation will be in effect from now on:
$$
\Gamma=\underset{1\leq i\leq
m}\bigcap(\Phi\setminus[\al_i,\rightarrow,\be_i))=\Phi\setminus(\underset{1\leq
i\leq m}\bigcup [\al_i,\rightarrow,\be_i)),
$$
$$
H_i=M(\Phi\setminus [\al_i,\rightarrow,\be_i)),\quad 1\leq i\leq
m,
$$
$$
H=M(\Gamma)=\underset{1\leq i\leq m}\bigcap H_i,
$$
$$
\widehat{\lm_i}:H_i\to F^*,
$$
the group homomorphism that extends $\lm_i:M_{\al_i\be_i}\to F^*$
(we are identifying $M_{\al_i\be_i}$ with $R^+$ via (\ref{rma}))
and is trivial on all $M_{\al\be}$ with $(\al,\be)\in
\Phi\setminus [\al_i,\rightarrow,\be_i]$,
$$
\lm=\widehat{\lm_1}|_{H}\cdots \widehat{\lm_m}|_{H},
$$
$$
W=Fw,
$$
a 1-dimensional $H$-module acted upon by $H$ via $\lm$.
}
\end{notation}

\begin{theorem}\label{monom} $V(D,f)\cong
\ind_H^M W$.
\end{theorem}

\begin{proof} We have $V(D,f)=Y\otimes V(\al_1,\be_1,\lm_1)$, where
$$Y=V(\al_m,\be_m,\lm_m)\otimes\cdots \otimes V(\al_2,\be_2,\lm_2),\quad
V(\al_1,\be_1,\lm_1)=\ind_{H_1}^M Z$$ and $Z=Fz$ is acted upon
$H_1$ via $\widehat{\lm_1}$. By induction, $Y\cong \ind_{H_0}^M
T$, where $H_0=H_2\cap\cdots\cap H_m$ and $T=Ft$ is acted upon
$H_0$ via $\widehat{\lm_2}|_{H_0}\cdots \widehat{\lm_m}|_{H_0}$.
Since $H_0H_1=M$ and $H_0\cap H_1=H$, Mackey Tensor Product
Theorem (cf. \cite[Theorem 2.1]{Sz}) yields $V(D,f)\cong\ind_H^M
W$.
\end{proof}

\begin{definition} Let $\al<\be$ and  $\ga<\de$ be in $\Lambda$. We
say that the intervals $(\al,\be)$ and $(\ga,\de)$ of $\Lambda$
are \emph{nested} if $\al<\ga<\de<\be$ or $\ga<\al<\be<\de$, and
\emph{overlapping} if $\al<\ga<\be<\de$ or $\ga<\al<\de<\be$.
\end{definition}

\begin{definition} We say that $D$ is \emph{nested} (resp. \emph{non-overlapping}) if
the open intervals (resp. none of the open intervals) of $\Lambda$
determined by the elements of $D$ are nested (resp. overlapping). 
\end{definition}

\begin{definition} We say that a disjoint union $D=D_1\cup\dots\cup D_n$ is a \emph{disconnection}
if each $D_i$ is non-empty and, for each $1\leq i<n$,
the conditions $(\al,\be)\in D_i$ and $(\ga,\de)\in D_{i+1}$ imply $\be\leq\ga$.
\end{definition}

\begin{notation}
Given a group $G$ and $G$-modules $U$ and $V$, we set
$$(U,V)_G=\dm_F\Hom_G(U,V).$$
\end{notation}


\begin{theorem}\label{nest} Suppose $D$ is nested. Then $V(D,f)$
is irreducible and
$$
(V(D,f),V(D,f))_M=1.
$$
\end{theorem}

\begin{proof} The fact that $D$ is nested translates as follows:
$$ \al_1<\al_2<\dots<\al_m<
\be_m<\dots<\be_2<\be_1.
$$

By Theorem \ref{schro}, $V(\al_1,\be_1,\lm_1)$ is an irreducible
$M$-module whose restriction to $M^{\al_1\be_1}$ remains
irreducible and $\End_{M^{\al_1\be_1}} V(\al_1,\be_1,\lm_1)=F$. By Lemma \ref{forma}
and Theorem \ref{cli4}, if $W$ is any irreducible $M$-module acted
upon trivially by $M^{\al_1\be_1}$ then $W\otimes
V(\al_1,\be_1,\lm_1)$ is an irreducible $M$-module and
$$
(W\otimes V(\al_1,\be_1,\lm_1),W\otimes V(\al_1,\be_1,\lm_1))_M=(W,W)_M.
$$

We have the decomposition
$$M=M([\al_1,\leftarrow))M((\rightarrow,\be_1])\rtimes
M((\al_1,\be_1)),$$ where
$M([\al_1,\leftarrow))M((\rightarrow,\be_1])$, and hence
$M^{\al_1\be_1}$, acts trivially on each $V(\al_i,\be_i,\lm_i)$,
$1<i\leq m$. Moreover, $W=V(\al_m,\be_m,\lm_m)\otimes\cdots\otimes
V(\al_2,\be_2,\lm_2)$ is a nested module of the McLain group
$M((\al_1,\be_1))$. By induction, $W$ is an irreducible
$M((\al_1,\be_1))$-module, where $(W,W)_{M((\al_1,\be_1))}=1$, and therefore $W$ is an irreducible
$M$-module acted upon trivially by $M^{\al_1\be_1}$ and satisfying $(W,W)_M=1$.
The result follows.
\end{proof}

\begin{theorem}\label{nono} Suppose $D$ is non-overlapping. Then $V(D,f)$
is irreducible.
\end{theorem}

\begin{proof} By hypothesis $D$ has a disconnection $D=D_1\cup\dots\cup D_n$, where each $D_i$ is nested.
Let $(\al_i,\be_i)$ be the outermost interval of $D_i$, so that
$\be_i\leq\al_{i+1}$ for each $1\leq i<n$. Let $f_i$ be the
restriction of $f$ to each $D_i$. By Theorem
\ref{nest}, each $V(D_i,f_i)$ is an irreducible
$M([\al_i,\be_i])$-module satisfying
$(V(D_i,f_i),V(D_i,f_i))_{M([\al_i,\be_i])}=1$,
so $V(D,f)$ is an irreducible
$M$-module by Corollary~\ref{cli6}.
\end{proof}

\begin{lemma}\label{redu} Let $H\leq G$ be groups and let $W$ be
an $H$-module. Let $I$ be a subgroup of $G$ properly containing
$H$. Suppose the action of $H$ on $W$ is extendible to~$I$ and call
the resulting $I$-module $W_1$. Then $\ind_H^G W$ is reducible.
\end{lemma}

\begin{proof} Let $T=Ft$ be the trivial $H$-module and let $P$ be the permutation
$I$-module associated to the coset space $I/H$. Then
$$
\begin{aligned}
\ind_H^G W &\cong \ind_I^G \ind_H^I W\\
&\cong \ind_I^G\ind_H^{I} (T\otimes W)\\
&\cong\ind_I^G \ind_H^{I} (T\otimes \res_H^I W_1)\\
&\cong\ind_I^G ((\ind_H^{I} T)\otimes W_1)\quad \text{(cf. \cite[\S 2]{Sz})}\\
&\cong\ind_I^G (P\otimes W_1).\\
\end{aligned}
$$
Since a permutation module of dimension $>1$ (finite or infinite)
is always reducible, so is $\ind_H^G W$.
\end{proof}

In regards to Lemma \ref{redu}, see the examples given in \cite[\S 5]{Sz}.

\begin{theorem}\label{yeso} Suppose $D$ is overlapping. Then $V(D,f)$
is reducible.
\end{theorem}

\begin{proof} It suffices to prove the result when $|D|=2$. We have $V(D,f)\cong
\ind_H^M W$, as in Theorem \ref{monom}.

Note that $M_{\al_1\al_2}\cap H=1$. Moreover, $M_{\al_1\al_2}$ normalizes
$H$ and stabilizes $\lm$. Therefore $\lm$ is extendible to
$I=H\rtimes M_{\al_1\al_2}$ via any map
$$
 h(1+re_{\al_1\al_2})\mapsto \lm(h)\chi(r)
$$
with $\chi\in\Hom(R^+,F^*)$. Now apply Lemma \ref{redu}.
\end{proof}

\begin{theorem}\label{basirr} $V(D,f)$ is irreducible if and only if
$D$ is non-overlapping.
\end{theorem}

\begin{proof} Immediate consequence of Theorems \ref{nono} and
\ref{yeso}.
\end{proof}

\section{Tools required to study the decomposition of a basic module}

\begin{theorem}\label{fund1} Let $H\unlhd I\leq G$ be groups and
let $W$ be an irreducible $H$-module stabilized by $I$. Given any
$x\in G$ let ${}^x W$ be the vector space $W$ when viewed as an
$H\cap xHx^{-1}$-module via
$$
h\cdot w=(x^{-1}hx)w.
$$
Suppose that
$$
(W,W)_H=1
$$
and that given any $x\in G\setminus I$, we have
$$
(W,{}^x W )_{H\cap xHx^{-1}}=0.
$$
Then
$$
(\ind_H^I W,\ind_H^I W)_I=[I:H]=(\ind_H^G W,\ind_H^G W)_G.
$$
\end{theorem}

\begin{note}{\rm The first equality is trivial. Under more restrictive hypothesis
the second equality would be a routine application of Frobenius
reciprocity, twice, with Mackey Decomposition Theorem used in
between. However, $(X,Y)_S$ need not equal $(Y,X)_S$, in general,
and this property would be required twice in that argument, first
for $S=H$ and then for $S=H\cap xHx^{-1}$. The second use could be
avoided if we modified our hypotheses, but the first could not.

The alternative argument presented below, which goes inside of the
proof of the Mackey Decomposition Theorem, will suffice for our
purposes.}
\end{note}

\begin{proof} Since $H\unlhd I$ and $I$ stabilizes $W$, we see
that $\res_H^I\ind_H^I W$ is the direct sum of $[I:H]$ copies of
$W$, so by Frobenius reciprocity
$$
(\ind_H^I W,\ind_H^I W)_I=(W,\res_H^I\ind_H^I W)_H=[I:H].
$$

Let $E$ be a system of representatives for the $(H,H)$-double
cosets in $G$. Since $H\unlhd I$, we may assume that $E$ contains
a system, say $E_0$, of representatives for $I/H$. Moreover, for
each $x\in E$, let $T_x$ be a system of representatives,
including~1, for the left cosets of $H\cap xHx^{-1}$ in $H$. Then
$$
S=\{yx\,|\, x\in E,y\in T_x\}
$$
is a set of representatives for the left cosets of $H$ in $G$. From
$$
\ind_H^G W=\underset{g\in S}\bigoplus\, gW,
$$
we obtain
$$
\res_H^G\ind_H^G W=\underset{x\in E}\bigoplus\, U_x,
$$
where $U_x$ is the $H$-submodule of $\ind_H^G W$ given by
$$
U_x=\underset{y\in T_x}\bigoplus yx W.
$$
Suppose first $x\in E_0$. Then $T_x=\{1\}$ and $U_x=xW\cong W$, so
$$
(W,U_x)_H=1.
$$
Suppose next $x\in E\setminus E_0$. Then $yx\notin I$ for any
$y\in T_x$. Therefore
$$
(W,yxW)_{H\cap yxH(yx)^{-1}}=0,\quad y\in T_x,
$$
and fortiori
$$
(W,U_x)_H=0.
$$
Since $W$ is $H$-irreducible, and hence generated by a single
element as $H$-module, it follows that $\Hom_H(W,-)$ commutes with
the direct sum of $H$-modules. Thus, by above
$$
(W,\res_H^G\ind_H^G W)_H=[I:H],
$$
so, by Frobenius reciprocity,
$$
(\ind_H^G W,\ind_H^G W)_G=[I:H].
$$
\end{proof}

\begin{prop}\label{compl} Let $S\unlhd T$ be groups, where $[T:S]$ is finite and $\chr(F)\nmid [T:S]$.
Let $W$ be a completely reducible $S$-module.
Then $\ind_S^T W$ is a completely reducible $T$-module.
\end{prop}

\begin{proof} We have $W=\underset{x\in X}\bigoplus W_x$ with $W_x$ irreducible,
so $\ind_S^T W\cong \underset{x\in X}\bigoplus\ind_S^T W_x$. We
may thus restrict to the case when $W$ itself is irreducible.
Since $\res_S^T \ind_S^T W$ is the direct sum of conjugates of the
irreducible $S$-module~$W$, it follows that $\res_S^T \ind_S^T W$
is a completely reducible $S$-module. Since $[T:S]$
is finite and $\chr(F)\nmid [T:S]$, a sharpened version of
Maschke's theorem (cf. \cite[\S 10.20, Problem 8]{CR}) ensures
that $\ind_S^T W$ is a completely reducible $T$-module.
\end{proof}

\begin{definition}\label{asc} Let $T$ be a group. We say that a subgroup $S$ of $T$ is \emph{strongly ascendant} (relative to $\chr(F)$) if there is well-ordered
set $(X,\leq)$ (not to be confused with the order used in
\S\ref{mc}), with first element $x_0$ and last element $x_1$, as
well as subgroups $S_x$, $x\in X$, of $T$ such that: $S_{x_0}=S$;
$S_{x_1}=T$; $S_x$ is normal of finite index in $S_{x'}$, with
$\chr(F)\nmid [S_{x'}:S_x]$, for every $x\in X$, $x\neq x_1$, with successor
$x'\in X$; if $x\in X$, $x\neq x_0$, and $x$ has no immediate
predecessor in $X$, then $S_x=\underset{y<x}\cup S_y$.
\end{definition}

\begin{theorem}\label{fund2} Keep all of the hypotheses of Theorem
\ref{fund1} and assume, in addition, that $[I:H]$ is finite and
not divisible by $\chr(F)$. Then

(a) $\ind_H^I W$ is a completely reducible $I$-module of finite
length $\leq [I:H]$ and
$$
\ind_H^I W=m_1V_1\oplus\cdots\oplus m_t V_t,
$$
where $\{V_1,\dots,V_t\}$ is a full set of representatives for the
isomorphism classes of irreducible $I$-modules lying over $W$. Moreover,
each $m_i$ satisfies
$$
d(i)=(W,\res_H^I V_i)_H=m_i (V_i,V_i)_I,
$$
where $d(i)$ is the (finite) length of the homogeneous $H$-module $\res_H^I V_i$, that is,
$$
\res_H^I V_i\cong W\oplus\cdots\oplus W\; (d(i) \text{ summands}\,).
$$

(b) We have
$$
\ind_H^G W\cong m_1\ind_I^G V_1\oplus\cdots\oplus m_t \ind_I^G
V_t,
$$
where
$$
(\ind_I^G V_i, \ind_I^G V_i)_G=(V_i,V_i)_I,\quad 1\leq i\leq t,
$$
$$
(\ind_I^G V_i, \ind_I^G V_j)_G=0,\quad 1\leq i\neq j\leq
t,
$$
$$
(\ind_I^G V_i,\ind_H^G W)=m_i (V_i,V_i)_I=(W,\res_H^I V_i)_H=d(i),\quad
1\leq i\leq t.
$$

(c) Suppose that at least one of the following conditions
holds:

{\rm (C1)} There is a normal subgroup $N$ of $G$ contained in $H$
and an irreducible $N$-module $W_0$ such that $W$ lies over $W_0$
and $I=I_G(W_0)$.

{\rm (C2)} $G$ is finite, $\chr(F)\nmid |G|$ and $(V_i,V_i)_I=1$
for all $1\leq i\leq t$.

{\rm (C3)} $I$ is a strongly ascendant subgroup of $G$ and
$(V_i,V_i)_I=1$ for all $1\leq i\leq t$.

Then $\ind_I^G V_1,\dots,\ind_I^G V_t$ are irreducible
$G$-modules. In particular, $\ind_H^G W$ is a completely reducible $G$-module.
\end{theorem}

\begin{proof} We know from Proposition \ref{compl} that $\ind_H^I W$ is a completely reducible $I$-module.
Since $\res_H^I \ind_H^I W$ is the direct sum of $[I:H]$
conjugates of the irreducible $H$-module $W$, we see that
$\res_H^I \ind_H^I W$ has finite length $[I:H]$, whence $\ind_H^I
W$ has finite length $\leq [I:H]$. Thus
\begin{equation}
\label{tos} \ind_H^I W=m_1V_1\oplus\cdots\oplus m_t V_t,
\end{equation}
where $V_1,\dots,V_t$ are non-isomorphic irreducible $I$-modules
and each $m_i\geq 1$. Projecting $\ind_H^I W$ onto each component,
we see that
$$
(\ind_H^I W,V_i)_I\neq 0,\quad 1\leq i\leq t,
$$
so by Frobenius reciprocity
$$
(W,\res^I_H V_i)_H\neq 0,\quad 1\leq i\leq t.
$$
Therefore each $V_i$ lies over $W$. Conversely, if $V$ is an
irreducible $I$-module lying over $W$ then
$$
(W, \res^I_H V)_H\neq 0,
$$
so
$$
(\ind_H^I W,V)_I\neq 0
$$
by Frobenius reciprocity. Hence (\ref{tos}) implies the existence
of one (and only one) $i$ such that $1\leq i\leq t$ and
$$
(V_i,V)_I\neq 0.
$$
Since $V_i$ and $V$ are irreducible $I$-modules, Schur's Lemma
yields $V_i\cong V$. Thus $\{V_1,\dots,V_t\}$ is a full set of
representatives for the isomorphism classes of irreducible
$I$-modules lying over $W$. Frobenius reciprocity and (\ref{tos})
now give
\begin{equation}
\label{uva} (W,\res_H^I V_i)_H=(\ind_H^I W,V_i)_I=(m_i
V_i,V_i)_I=m_i(V_i,V_i)_I,\quad 1\leq i\leq t.
\end{equation}
Since $H\unlhd I$, with $[I:H]$ finite, and $W$ is an irreducible
submodule of $\res_H^I V_i$, Clifford's theory ensures that
$\res_H^I V_i$ is the direct sum $d(i)$ conjugates of $W$, where
$d(i)$ is the (finite) length of $\res_H^I V_i$ as an $H$-module.
But $W$ is $I$-invariant, so $\res_H^I V_i$ is the direct sum of
$d(i)$ copies of $W$. Thus
\begin{equation}
\label{melon} (W,\res_H^I V_i)_H=d(i),\quad 1\leq i\leq t.
\end{equation}

Let $S$ be any subgroup of $G$ containing $I$. From (\ref{tos}) we obtain

\begin{equation}\label{tos2}
\begin{aligned}
 \ind_H^S W &\cong\ind_I^S \ind_H^I W\\&\cong
\ind_I^S(m_1V_1\oplus\cdots\oplus m_t V_t)\\& \cong m_1\ind_I^S
V_1\oplus\cdots\oplus m_t \ind_I^S V_t.
\end{aligned}
\end{equation}
From Theorem \ref{fund1}, (\ref{tos}) and Schur's Lemma we get
\begin{equation}
\label{tos3} [I:H]=m_1^2(V_1,V_1)_I+\cdots+m_t^2(V_t,V_t)_I.
\end{equation}
Setting $V_i'=\ind_I^S V_i$ for $1\leq i\leq t$, Theorem
\ref{fund1} and (\ref{tos2}) yield
\begin{equation}
\label{tos4}
[I:H]=m_1^2(V_1',V_1')_S+\cdots+m_t^2(V_t',V_t')_S+\underset{i\neq
j}\sum m_im_j(V_i',V_j')_S.
\end{equation}
Now, by Frobenius reciprocity
$$
(V_i',V_i')_S=(V_i,\res^S_I V_i')_I,\quad
1\leq i\leq t.
$$
Since $V_i$ is an $I$-submodule of $\res^S_I V_i'$, it
follows that
\begin{equation}
\label{tos5} (V_i',V_i')_S\geq (V_i,V_i)_I,\quad
1\leq i\leq t.
\end{equation}
Combining (\ref{tos3}), (\ref{tos4}) and (\ref{tos5}) we infer
\begin{equation}
\label{tos6} (V_i',V_i')_S=(V_i,V_i)_I\text{ and
}(V_i',V_j')_S=0,\quad 1\leq i\neq j\leq t.
\end{equation}
Making use of (\ref{tos2}) and (\ref{tos6}) yields
\begin{equation}
\label{tos7} (V_i', \ind_H^S W)_S=(V_i', m_i
V_i')_S=m_i(V_i',V_i')_S=m_i(V_i,V_i)_I,\quad 1\leq i\leq t.
\end{equation}
Thus (\ref{uva}), (\ref{melon}) and (\ref{tos7}) give
$$
(\ind_I^S V_i, \ind_H^S W)_S=(W,\res_H^I V_i)_H=d(i), \quad 1\leq i\leq t.
$$

Suppose (C1) holds. Since each $V_i$ is an irreducible $I$-module
lying over the irreducible $N$-module $W_0$ and $I=I_G(W_0)$, it
follows from Clifford's Correspondence (cf. Theorem \ref{cl2})
that each $\ind_I^G V_i$ is an irreducible $G$-module.

Suppose (C2) holds.  Then, by Maschke's theorem, each $\ind_I^G
V_i$ is a completely reducible $G$-module. Moreover, by
(\ref{tos6}), we have  $(\ind_I^G V_i,\ind_I^G V_i)_G=1$, so each
$\ind_I^G V_i$ is irreducible.

Suppose (C3) holds. Let $(X,\leq)$ and $(S_x)_{x\in X}$ be as in
Definition \ref{asc} and set $V=V_i$, where $1\leq i\leq t$. Then
$\ind_I^{S_{x_0}} V=V$ is irreducible. Let $x_0<x$ and suppose
$\ind_I^{S_y} V$ is irreducible for every $y<x$.

\noindent{Case 1.} $x$ has an immediate predecessor $y$. Since
$$
\ind_I^{S_x} V\cong \ind_{S_{y}}^{S_x} \ind_I^{S_{y}} V,\quad \quad 1\leq i\leq t,
$$
it follows from Proposition \ref{compl} that $\ind_I^{S_x} V$ is a completely reducible $S_x$-module.
But, by above, $(\ind_I^{S_{x}} V,\ind_I^{S_{x}} V)_{S_{x}}=1$, so $\ind_I^{S_x} V$ is irreducible.

\noindent{Case 2.} $x$ has no immediate predecessor. We readily
see that there is a system of representatives $Z$ for the left
cosets of $I$ in $S_x$ such that $Z\cap S_y$ is a system of
representatives for the left cosets of $I$ in~$S_y$ for each
$y\leq x$. We have
$$
\ind_I^{S_x} V=\underset{z\in Z}\bigoplus zV.
$$
Let $u$ be an arbitrary non-zero element of $\ind_I^{S_x} V$ and let $v\in \ind_I^{S_x} V$. Then there is a finite subset $\hat{Z}$ of $Z$ such that
$$
u,v\in \underset{z\in \hat{Z}}\bigoplus zV.
$$
Since $S_x=\underset{y<x}\cup S_y$, there is $y<x$ such that $\hat{Z}\subseteq S_y$. Then
$$
u,v\in \underset{z\in S_y\cap Z}\bigoplus zV\cong\ind_I^{S_y} V.
$$
Since $\ind_I^{S_y} V$ is irreducible, $v\in FS_y\cdot u$, whence $v\in FS_x\cdot u$. This proves that $\ind_I^{S_x} V$ is irreducible.

By transfinite induction, $\ind_I^{S_x} V$ is irreducible for
every $x\in X$. Since $G=S_{x_1}$, the proof is complete.
\end{proof}

\begin{note}{\rm It is false, in general, that if $S$ is a
strongly ascendant subgroup of $T$ and $W$ is an irreducible
$S$-module then $\ind_S^T W$ is completely reducible. Indeed, let
$T$ be the direct product of countably many copies of any finite
non-trivial group $P$ and let $S$ be the trivial subgroup of $T$
with trivial $S$-module $W$. Suppose $\chr(F)\nmid |P|$. Then
$\ind_S^{S_i} W$ is completely reducible for every $i$ (where
$S_i$ is the direct product of $i$ copies of $P$), but $V=\ind_S^T
W$ is the regular module of the infinite group~$T$ and hence is
not completely reducible (the epimorphism $V\to F$ shows that a
supposed complement to the augmentation ideal must be trivial, but
$V$ has no trivial submodule).}
\end{note}

As a consequence of Theorem \ref{fund2} we obtain the following extension of a
well-known irreducibility criterion due to Mackey \cite{Ma}, originally proved in the context
of finite groups and finite dimensional modules over an algebraically closed field.

\begin{theorem}\label{mage} Let $H\leq G$ be groups and
let $W$ be an irreducible $H$-module satisfying
$$
(W,W)_H=1
$$
and
$$
(W,{}^x W )_{H\cap xHx^{-1}}=0,\quad x\in G\setminus H.
$$
Suppose at least one of the following conditions hold:

{\rm (D1)} There is a normal subgroup $N$ of $G$ contained in $H$
and an irreducible $N$-module $W_0$ such that $W$ lies over $W_0$
and $ H=I_G(W_0)$ (this is automatic if $H\unlhd G$).

{\rm (D2)} $G$ is finite and $\chr(F)\nmid |G|$.

{\rm (D3)} $H$ is a strongly ascendant subgroup of $G$.

Then $\ind_H^G W$ is an irreducible $G$-module and $(\ind_H^G W,\ind_H^G W)_G=1$.
\end{theorem}

\section{Preparation for Mackey theory}\label{prema}

\begin{notation}{\rm Let
$$
\Omega=\{(\al,\ga)\in\Phi\,|\,\exists\, (\al,\be),(\ga,\de)\in D\text{ such that
} \al<\ga<\be<\de\},
$$
$$
\Gamma_1=\Gamma\cup\Omega\text{ (disjoint union)}.
$$
}
\end{notation}

\begin{note}{\rm Since $D$ is always finite, so is the subset $\Omega$ of $\Phi$.}
\end{note}

\begin{lemma}\label{clonor} $\Gamma_1$ is a closed subset of
$\Phi$.
\end{lemma}

\begin{proof} This is essentially proven in \cite[Proposition 2]{A3}.
\end{proof}

\begin{notation}{\rm Let
$$
I=M(\Gamma_1).
$$
}
\end{notation}

\begin{theorem}\label{corin}


(a) $H$ is normal in $I$.


(b) $[I:H]=|R|^{|\Omega|}$. Thus, $I/H$ is a finite group provided $R$ is a finite ring.

(c) $I$ stabilizes $\lm$.

(d) Given any $g\in M\setminus I$ there is $h\in H\cap gHg^{-1}$
such that $\lm(h)\neq \lm(g^{-1}hg)$.

\end{theorem}

\begin{proof} (a) This is included in the proof of \cite[Proposition 2]{A3}.

(b) This follows from Proposition \ref{trax}.

(c) By Proposition \ref{co1} it suffices to show that, given $(\al,\ga)\in \Omega$ and $r\in R$,
the conjugate character of $\lm$ by $1+re_{\al\ga}$ equals $\lm$. By Proposition \ref{co1} this verification can be restricted
to
\begin{equation}
\label{veri}
\lm((1+re_{\al\ga})(1+se_{\pi\rho})(1-re_{\al\ga}))=\lm(1+se_{\pi\rho})
\end{equation}
for all $(\pi,\rho)\in \Gamma$ and $s\in R$. Since $(\al,\ga)\in \Omega$ there exist $(\al,\be),(\ga,\de)\in D$
such that
$$
\al<\ga<\be<\de.
$$
Two cases arise:

\noindent{\sc Case 1.} $\pi=\ga$. In this case
$$
(1+re_{\al\ga})(1+se_{\ga\rho})(1-re_{\al\ga})=(1+se_{\ga\rho})(1+rse_{\al\rho}),
$$
so we need to check that
$$
\lm(1+rse_{\al\rho})=1.
$$
This is automatically true of $(\al,\rho)\notin D$. But $(\al,\rho)$ cannot be in $D$, for in that case $\rho=\be$,
whence $(\ga,\be)\notin \Gamma$.

\noindent{\sc Case 2.} $\rho=\al$. In this case
$$
(1+re_{\al\ga})(1+se_{\pi\al})(1-re_{\al\ga})=(1+se_{\pi\al})(1-sre_{\pi\ga}),
$$
so we need to check that
$$
\lm(1-sre_{\pi\ga})=1.
$$
This is automatically true of $(\pi,\ga)\notin D$. But $(\pi,\ga)$ cannot be in $D$, for in that case
$(\pi,\al)\notin \Gamma$.

(d) This is essentially contained in the proof of \cite[Proposition 1]{A3}. However, due to the its critical role
and technical nature, we reproduce Andr\'{e}'s argument, suitably modified to our purposes.

Let $g\in M\setminus I$. Then $g=1+x$, where
$$
x=\underset{(\al,\be)\in\Phi}\sum x_{\al\be}e_{\al\be},
$$
where all $x_{\al\be}\in R$ and almost all of them are equal to 0. Since $g\notin I$, the set
$$
A=\{(\al,\be)\in \Phi\setminus\Gamma_1\,|\, x_{\al\be}\neq 0\}
$$
is non-empty. To any $(\al,\be)\in A$ there corresponds a unique $(\al,\ga)\in D$. Thus,
$$
D_1=\{(\al,\ga)\in D\,|\,\exists\; (\al,\be)\in \Phi\setminus\Gamma_1 \text{ such that }x_{\al\be}\neq 0\}
$$
is non-empty. Choose $(\al,\ga)\in D_1$ with $\ga$ as large as large as possible. This can be done, because $D_1$ is a finite,
totally ordered, non-empty set.

It follows that if $(\rho,\si)\in D$, $\ga<\si$, and $x_{\rho\tau}\neq 0$ for some $\rho<\tau<\si$ (so that $(\rho,\tau)\notin\Gamma$),
then necessarily $(\rho,\tau)\in\Omega$ (for otherwise $(\rho,\tau)\notin\Gamma_1$, against the choice of $(\al,\ga)$) (*).

For this choice of $(\al,\ga)\in D_1$ there exists $(\al,\be)\in \Phi\setminus\Gamma_1$ such that $x_{\al\be}\neq 0$.

We claim that for all $r\in R$, we have
$$
g(1+re_{\be\ga})g^{-1}\in H\cap gHg^{-1}.
$$
Indeed, since $(\al,\ga)\in D$ and $(\al,\be)\notin\Gamma$, we have $\al<\be<\ga$. Moreover, since $(\al,\be)\notin\Omega$,
for any $\de\in\Lambda$
\begin{equation}\label{bela}
\ga<\de\Rightarrow (\be,\de)\notin D.
\end{equation}
If $(\be,\ga)\notin\Gamma$ there would be $\de\in\Gamma$ such that $\be<\ga<\de$ and~$(\be,\de)\in D$, against (\ref{bela}).
This shows that $(\be,\ga)\in \Gamma$, so $1+re_{\be\ga}\in H$, whence $g(1+re_{\be\ga})g^{-1}\in gHg^{-1}$.

Next we show that
$$
g(1+re_{\be\ga})g^{-1}=(1+x)(1+re_{\be\ga})(1+x)^{-1}=(1+x)(1+re_{\be\ga})(1+y)\in H,
$$
where $(1+x)^{-1}=1+y$ for a unique $y\in J$ (the nil ring $J$ was defined in \S\ref{mc}). Now
\begin{equation}\label{piw}
(1+re_{\be\ga})(1+y)=1+y+re_{\be\ga}+re_{\be\ga}y.
\end{equation}
Since $(\be,\ga)\in \Gamma$, it follows that $(\be,\de)\in\Gamma$ for every $\de\in\Lambda$ such that $\ga\leq\de$.
Thus, the multiplication in $J$ implies
\begin{equation}\label{bela2}
1+re_{\be\ga}+re_{\be\ga}y\in H.
\end{equation}
Multiplying (\ref{piw}) on the left by $(1+x)$ we find
\begin{equation}\label{piw2}
g(1+re_{\be\ga})g^{-1}=1+re_{\be\ga}+re_{\be\ga}y+xre_{\be\ga}+xre_{\be\ga}y.
\end{equation}
Suppose, if possible, that for some $(\rho,\pi)\notin\Gamma$, the $(\rho,\pi)$ coefficient of $xre_{\be\ga}$ is not~0.
Then $\pi=\ga$ and $\rho<\be$, in which case the coefficient is $x_{\rho\be}r\neq 0$. Since $(\rho,\ga)=(\rho,\pi)\notin\Gamma$,
there exists $(\rho,\si)\in D$ with $\rho<\ga<\si$. As $x_{\rho\be}\neq 0$ and $\rho<\be<\ga<\si$, (*) implies $(\rho,\be)\in\Omega$.
Then there exists $(\be,\de)\in D$ such that $\rho<\be<\si<\de$. Since $\ga<\si$, we infer $\ga<\de$, against (\ref{bela}).
This proves that $xre_{\be\ga}$ is in the $R$-span of $\Gamma$. As above, this implies that $xre_{\be\ga}y$ is in the $R$-span of $\Gamma$.
Combining this with (\ref{bela2}) proves the claim.

We next claim that there is $r\in R$ such that $h_r=g(1+re_{\be\ga})g^{-1}\in H\cap gHg^{-1}$ satisfies $\lm(h_r)\neq \lm(g^{-1}h_r g)$.
Since $\al\neq \be$ and $(\al,\ga)\in D$, it follows that $(\be,\ga)\notin D$, whence
$$
\lm(g^{-1}h_r g)=\lm(1+re_{\be\ga})=1,\quad r\in R.
$$
Thus, we are reduced to showing the existence of $r\in R$ such that
\begin{equation}\label{piw3}
\lm(h_r)\neq 1.
\end{equation}
Now $(\al,\ga)\in D$ means $(\al,\ga)=(\al_i,\be_i)$ for a unique $1\leq i\leq m$. As $\lm_i$ is right primitive and $x_{\al\be}\neq 0$,
there is $r\in R$ such that $\lm_i(x_{\al\be}r)\neq 1$, that is,
\begin{equation}\label{piw4}
\lm(1+x_{\al\be}re_{\al\ga})\neq 1.
\end{equation}
We claim that (\ref{piw3}) holds for this choice of $r$. Indeed, when we express $h_r$, namely (\ref{piw2}), in the canonical form
(\ref{uni2}), the non-trivial factors must be of one of the following forms:

(i) $1+re_{\beta \gamma}$,

(ii) $1+ry_{\gamma \delta}e_{\beta \delta}$, with $\gamma<\delta$,

(iii) $1 + x_{\alpha \beta}r e_{\alpha \gamma}$,

(iv) $1 + x_{\rho \beta} r e_{\rho \gamma}$, with $\alpha \ne
\rho<\beta$, or

(v) $1+x_{\rho \beta}r y_{\gamma \delta} e_{\rho \delta}$, with
$\rho < \beta < \gamma < \delta$.

Since $\lambda$ is a linear character of $H$, the value of
$\lambda$ on $h_r$ is equal to the
product of its values on the non-trivial factors above. In view of (\ref{piw4}), it remains to show that
$\lm$ has value 1 on all non-trivial factors of type different from (iii). This
is clear for types (i), (ii) and (iv) since $(\rho,\ga),(\be,\ga)\notin D$ (as $\be,\rho\neq \al$) and $(\be,\de)\notin D$ (by (\ref{bela})).
Suppose, if possible, that a factor of type (v) is not trivial (so, in particular, $x_{\rho \beta}\neq 0$)
with $(\rho,\delta)\in D$. Then, by (*), $(\rho,\be)\in\Omega$. Thus there is $\epsilon\in\Lambda$ such $(\be,\epsilon)\in D$
and $\rho<\be<\delta<\epsilon$. But $\gamma<\epsilon$, so (\ref{bela}) is contradicted.  
\end{proof}





\section{Decomposition of a basic module}\label{aqui}

\begin{theorem}\label{inv} We have
$$(V(D,f),V(D,f))_M=[I:H]=|R|^{|\Omega|}.$$
\end{theorem}

\begin{proof} Immediate consequence of Theorems \ref{monom}, \ref{fund1} and
\ref{corin}.
\end{proof}

\begin{note}{\rm Theorem \ref{inv} is valid for arbitrary $(\Lambda,\leq)$ and $R$, not necessarily finite.}
\end{note}

\begin{theorem}\label{coni} The following conditions are equivalent:

(a) $V(D,f)$ is irreducible.

(b) $D$ is non-overlapping.

(c) $\Omega=\emptyset$.

(d) $I=H$.

(e) $(V(D,f),V(D,f))_M=1$.
\end{theorem}

\begin{proof} This follows from Theorems \ref{basirr} and \ref{inv}.
\end{proof}

\begin{definition}\label{skipi} We will say that $F$ is a \emph{splitting field
for $I$ over $\lm$} if $(V,V)_I=1$ for every irreducible
$I$-module lying over $\lm$.
\end{definition}

\begin{note}\label{pie}{\rm We know from Theorem \ref{fund2} that an irreducible
$I$-module $V$ lies over~$\lm$ if and only if $V$ is a constituent
of the finite dimensional, completely reducible module $\ind_H^I
W$. Thus, there are finitely many irreducible $I$-modules lying
over $\lm$. Since $[I:H]$ and $\dm(W)$ are finite, every
irreducible $I$-module lying over $\lm$ is finite dimensional.
Thus there is a finite extension $K$ of $F$ that is a splitting
field for $I$ over $\lm$ (let $K$ be the subfield of
$\overline{F}$ generated by the entries of the matrix
representations associated to each isomorphism type of irreducible
$I$-module lying over $\lm$). Thus, there is not much loss of
generality in assuming that $F$ itself is a splitting field for
$I$ over $\lm$.}
\end{note}

\begin{note}\label{uop}{\rm
Suppose that $\lm$ is extendible to $I$ (see Theorem \ref{ext} for
the exact conditions when this happens). Since $(W,W)_H=1$, it
follows from Lemma \ref{forma} and Theorem \ref{cli4} that $F$ is
a splitting field for $I$ over $\lm$ provided $F$ is a splitting
field for $I/H$ in the usual sense. In Example \ref{yeka2}, $\lm$
is not extendible to~$I$ but, nevertheless, a splitting field for
$I/H$ is a splitting field for $I$ over $\lm$. This implication
may be true in general, a matter that will not be discussed
further. }
\end{note}

\begin{theorem}\label{xasc} Let $\al=\min\{\al_i\,|\, 1\leq i\leq m\}$ and
$\be=\max\{\be_i\,|\, 1\leq i\leq m\}$. Suppose the closed interval
$[\al,\be]$ of $\Lambda$ is well-ordered under $\geq$, the inverse order
of~$\leq$. Assume, in addition, that
$R$ is finite. Then $I$ is a strongly ascendant subgroup of $M$.
\end{theorem}


\begin{proof} Let $\Psi$ be a closed subset of
$\Phi$ satisfying the following conditions:

(A1) There is a first $\de\in (\al,\be]$ under $\geq$ such that
$M((\rightarrow,\de])\subseteq M(\Psi)$.

(A2) If $\ga=\de'$ is the immediate successor of $\de$ in $[\al,\be]$ under $\geq$, then
$$
\Psi^\sharp=\{(\si,\ga)\in\Phi\,|\, (\si,\ga)\notin\Psi\}
$$
is a finite subset of $\Phi$.

An example is $\Psi=\Gamma_1$, in which case (A1) holds with $\de=\be$ and $|\Psi^\sharp|\leq 1$.

In any case, either $\Psi^\sharp=\emptyset$, and we set
$\Psi'=\Psi$, or else
\begin{equation}
\label{punta}
\Psi^\sharp=\{(\si_1,\ga),\dots,(\si_n,\ga)\},
\end{equation}
$$
\si_1<\cdots<\si_n,
$$
in which case the following are all closed subsets of $\Phi$, each normal in the next:
$$
\Psi(0)=\Psi\subset \Psi(1)=\Psi\cup\{(\si_1,\ga)\} \subset\cdots\subset \Psi(n)=\Psi\cup \{(\si_1,\ga),\dots,(\si_n,\ga)\}.
$$
Moreover, by Proposition \ref{trax},
$$
[M(\Psi(i+1)):M(\Psi(i))]=|R|,\quad 0\leq i<n,
$$
where $\chr(F)\nmid |R|$ by Lemma \ref{cara}, and $\Psi'=\Psi(n)$ satisfies (A1), (A2) as well as
\begin{equation}
\label{adentro}
M((\rightarrow,\ga])\subseteq M(\Psi').
\end{equation}
Define the family $(\Psi_\ga)_{\ga\in [\al,\be]}$ of closed subsets of $\Psi$ as follows:
$$
\Psi_\be=\Gamma_1,\,  \Psi_{\ga}=(\Psi_\de)'\text{ if
}\ga=\de',\,\Psi_\ga=\underset{\de>\ga}\cup \Psi_\de\text{ if
}\ga\text{ has no immediate } \geq\!-\text{predecessor}.
$$
This is possible because, as $D$ is finite, so is (\ref{punta})
when $\Psi=\Psi_\de$ for any $\de\in (\al,\be]$. By construction,
$\Psi_\de$ is subnormal in $\Psi_{\de'}$, with
$[M(\Psi_{\de'}):M(\Psi_\de)]$ finite and not divisible by
$\chr(F)$ for every $\de\in (\al,\be]$. Moreover, (\ref{adentro})
implies $M=M(\Psi_\al)$.
\end{proof}

\begin{theorem}\label{nuc} Suppose $R$ is finite and $F$ is a splitting field for $I$ over
$\lm$. Let
$$
\ind_H^I W=m_1V_1\oplus\cdots\oplus m_t V_t,
$$
be the decomposition of $\ind_H^I W$ ensured by part (a) of
Theorem \ref{fund2} (which applies, due to Lemma \ref{cara} and
parts (a), (b) and (c) of Theorem \ref{corin}), where
$\{V_1,\dots,V_t\}$ is a full set of representatives for the
isomorphism classes of irreducible $I$-modules lying over $W$, and
each $m_i$ satisfies
$$
(W,\res_H^I V_i)_H=\dm(V_i)=m_i.
$$
Assume, in addition, that $I$ is a strongly ascendant subgroup of $M$. Then
$$
V(D,f)\cong m_1\ind_I^M V_1\oplus\cdots\oplus m_t \ind_I^M V_t,
$$
where $\ind_I^M V_1,\dots,\ind_I^M V_t$ are non-isomorphic irreducible $M$-modules, and
$$
(\ind_I^{M} V_i, V(D,f))_M=\dm(V_i)=m_i, \quad 1\leq i\leq t.
$$
\end{theorem}

\begin{proof} Use Theorems \ref{monom}, \ref{fund2} and
\ref{corin}(d).
\end{proof}

\begin{note}\label{ascbueno}{\rm Suppose $R$ is finite. In view of Theorem \ref{xasc}, all we have to do to ensure
that every $I$ is a strongly ascendant subgroup of $M$ is to start
with any well-order and impose its inverse on $\Lambda$. This is
automatic when $\Lambda$ is finite, which is just a very special,
albeit important, case.

Observe that $I$ need not be a strongly ascendant subgroup of $M$ in general.
Suppose, for instance, that $\Lambda$ is a closed interval
$[\al,\be]$ and $\be$ is an accumulation point when $[\al,\be]$ is
given the topology of the strict order $<$ associated to $\leq$
(cf. \cite[Chapter 1, Problem I]{K}). If $D=\{(\al,\be)\}$, then
$I=H$ is self-normalizing.
}
\end{note}

\begin{note}
{\rm Suppose $R$ is finite, $M=U_n(R)$ and $F$ is a splitting
field for (the finite group) $I$. Then condition (C2) of Theorem
\ref{fund2} is satisfied, and the conclusion of Theorem~\ref{nuc}
follow automatically, that is, without resorting to strongly
ascendant subgroups.}
\end{note}

\begin{note}{\rm Let
$$
\Gamma_0=\Gamma\setminus\{(\sigma,\tau)\in\Gamma\,|\,\exists\;
i\text{ such that }1\leq i\leq m\text{ and }
\al_i<\sigma<\tau<\be_i\}.
$$
It is easy to see that $N=M(\Gamma_0)$ satisfies 
$$
N=\mathrm{core}_M(H).
$$
By Theorem \ref{corin}(c), we have $I\subseteq I_M(\lm|_N)$,
although equality is not necessarily true in general. However, it
does occur, occasionally. Suppose, for instance, that, if $1\leq
i<m$, then $\al_{i+1}$ is the only element of $\Lambda$ satisfying
$\al_i<\al_{i+1}< \be_i$. Then $I=I_M(\lm|_N)$ (the case $m=1$ is
treated in \S\ref{elemeM}). Whatever the example, suppose that
$I=I_M(\lm|_N)$. Then condition (C1) of Theorem \ref{fund2} is
satisfied. Thus, if $R$ is finite and $F$ is a splitting field for
$I$ over $\lm$, then the conclusion of Theorem~\ref{nuc} is
achieved without involving strongly ascendant subgroups.}
\end{note}

\begin{definition}\label{trip} We will say that $D$ has a special triple if it contains a subset $\{(\al_1,\be_1), (\al_2,\be_2), (\al_3,\be_3)\}$
satisfying:
$$
\al_1<\al_2<\al_3=\be_1<\be_2<\be_3.
$$
\end{definition}

\begin{prop}\label{barb} $[\Gamma_1,\Gamma_1]\cap D=[\Omega,\Omega]\cap D$. Moreover, these
are non-empty if and only if $D$ contains a special triple.
\end{prop}

\begin{proof} Suppose $(\al,\be)\in [\Gamma_1,\Gamma_1]\cap D$. Then, by definition, there is a chain
$$(\ga_1,\ga_2),\dots,(\ga_{n-1},\ga_n)\in \Gamma_1,\quad n\geq 3,$$ such that
$\ga_1=\al$ and $\ga_n=\be$. Since $\Gamma_1\subseteq\Phi$, we have
$$
\al=\ga_1<\ga_2<\cdots<\ga_{n-1}<\ga_n=\be.
$$
In particular, $\al<\ga_2<\be$, so $(\al,\ga_2)\notin\Gamma$ and therefore $(\al,\ga_2)\in\Omega$. By definition,
there exists $(\al_2,\be_2)\in D$ such that
$$
\al<\al_2<\be<\be_2,\quad \ga_2=\al_2.
$$
From
$$
\al_2<\ga_3\leq \be<\be_2,
$$
we infer $(\al_2,\ga_3)\notin\Gamma$, hence $(\al_2,\ga_3)\in\Omega$. By definition,
there exists $(\al_3,\be_3)\in D$ such that
$$
\al_2<\al_3<\be_2<\be_3,\quad \ga_3=\al_3.
$$
If $\al_3=\be$ then
$$
\al<\al_2<\al_3=\be<\be_2<\be_3
$$
and we are done. Otherwise,
$$
\al_3<\ga_4\leq \be <\be_3,
$$
implies, as before,
the existence of $(\al_4,\be_4)\in D$ such that
$$
\al_3<\al_4<\be_3<\be_4,\quad \ga_4=\al_4.
$$
If $\al_4=\be$ then
$$
\al<\al_2<\al_3<\al_4=\be<\be_2<\be_3<\be_4
$$
and we are done. Otherwise continue this process to obtain the desired result.
\end{proof}

\begin{theorem}\label{ext} The following conditions are equivalent:

(a) $\lm$ is extendible to $I$.

(b) $[\Gamma_1,\Gamma_1]\cap D=\emptyset$.

(c) $[\Omega,\Omega]\cap D=\emptyset$.

(d) $D$ has no special triple.
\end{theorem}

\begin{proof} Suppose first $(\al,\be)\in [\Gamma_1,\Gamma_1]\cap D$. By Proposition \ref{co2}, $M_{\al\be}\subseteq [I,I]$,
so any group homomorphism $I\to F^*$ is trivial on $M_{\al\be}$, whence $\lm$ is not extendible to $I$.

Suppose next $[\Gamma_1,\Gamma_1]\cap D=\emptyset$. Then $D\subseteq \Gamma_1\setminus [\Gamma_1,\Gamma_1]$. We certainly
have a group homomorphism from the external direct product of all $M_{\al\be}$, $(\al,\be)\in\Gamma_1\setminus [\Gamma_1,\Gamma_1]$,
$$\mu:\underset{(\al,\be)\in\Gamma_1\setminus [\Gamma_1,\Gamma_1]}\prod
M_{\al\be}\to F^*$$ that agrees with $\lm_i$ on $M_{\al_i\be_i}$ for every $1\leq i\leq m$. It follows from Proposition \ref{co2}
that $\mu$ gives rise to an extension of $\lm$ to $I$.

This proves the equivalence between (a) and (b). Now apply Proposition \ref{barb}.
\end{proof}

\begin{theorem}\label{tx} Suppose $|R|$ is finite, $F$ is a splitting field for $S=I/H$ and $D$ has no special triple.
Assume, in addition, that $I$ is a strongly ascendant subgroup of~$M$. Then $V(D,f)$ has the
following decomposition as the direct sum of irreducible
non-isomorphic $M$-modules with indicated multiplicities:
$$
V(D,f)\cong \underset{\chi\in\Irr(S)}\bigoplus
\chi(1)\;\ind_I^M (U_\chi\otimes W_1),
$$
where $U_\chi$ is an irreducible $S$-module -viewed as an $I$-module- affording $\chi$, and $W_1$ is the vector space $W$
acted upon via an extension of $\lm$ to $I$.
\end{theorem}

\begin{proof} Use Note \ref{uop} as well as Theorems \ref{cli4}, \ref{nuc} and \ref{ext}.
\end{proof}

\begin{cor}\label{yeka} Suppose that $R$ is finite, $D$ has no special triple, and $S=I/H$ is abelian (this is equivalent to $[\Omega,\Omega]\subseteq \Gamma)$).
Assume, in addition, that $I$ is a strongly ascendant subgroup of $M$. Then $V(D,f)$ has the
following decomposition into non-isomorphic $M$-modules:
$$
V(D,f)\cong \underset{\chi\in\Irr(S)}\bigoplus \ind_I^M
W_1(\chi),
$$
where $W_1(\chi)$ is the vector space $W$ acted upon $I$ via the only extension of $\lm$ to $I$
that satisfies $1+re_{\sigma\tau}\mapsto \chi(r)$ for all $(\sigma,\tau)\in\Omega$.
\end{cor}

\begin{proof} Use Lemma \ref{cara2} and Theorem \ref{tx}.
\end{proof}

\begin{definition}\label{ti1} We will say that
$D$ is overlapping of type 1 provided the following conditions are
satisfied:

$\bullet$ $\al_1<\dots<\al_m$ (this can always be arranged), where
$m\geq 2$.

$\bullet$  The intervals $(\al_i,\be_i)$, $(\al_j,\be_j)$ overlap if and only if $|i-j|=1$;

$\bullet$ No $\be_i$ equals an $\al_j$ (this is automatic if
$|D|=2$).
\end{definition}

Let $\Lambda=\N$ under its usual order. Then $\{(1,3), (2,5), (4,7), (6,8)\}$ is an overlapping subset of $\Phi$ of type 1,
whereas $\{(1,3), (2,4), (3,5)\}$ and $\{(1,4), (2,5), (3,6)\}$ are not.

As the following result indicates, a family of examples to
which Corollary \ref{yeka} applies is given by the overlapping subsets of $\Phi$ of type 1.

\begin{lemma}\label{yeka5} Suppose $D$ is an
overlapping subset of $\Phi$ of type~1. Then $D$ has no special triples and $I/H$ is abelian.
\end{lemma}

\begin{proof} By Definition \ref{ti1}, we have $\Omega=\{(\al_1,\al_2),\dots, (\al_{m-1},\al_m)\}$,
$[\Omega,\Omega]\subseteq \Gamma$ and $D$ has no special triples.
\end{proof}


\begin{exa}\label{yeka2} Suppose $R$ is finite, $D$ satisfies
$$
\al_1<\al_2<\cdots<\al_m=\be_1<\be_2<\cdots<\be_m,\quad m\geq 3,
$$
and $F$ is a splitting field for $U_{m-2}(R)$. Assume, in
addition, that $[\al_1,\be_m]$ is well-ordered by $\geq$. Then
$V(D,f)$ has the following decomposition as the direct sum of
irreducible non-isomorphic $M$-modules with indicated
multiplicities:
$$
V(D,f)\cong \underset{\chi\in\Irr(B)}\bigoplus
\dm(U_\chi)\dm(V)\;\ind_I^M (U_\chi\otimes V).
$$
Here $I=T\rtimes B$ for suitable subgroups $B\cong U_{m-2}(R)$ and
$T$ described below; each $U_\chi$ is an irreducible $B$-module
-viewed as an $I$-module- affording $\chi$; $V$ is an irreducible
$I$-module of dimension $|R|^{m-2}$, as described below; each
$U_\chi\otimes V$ is an irreducible $I$-module. Moreover, when $m=3$ we have $V(D,f)\cong |R|\ind_I^M V$.
\end{exa}

\begin{proof} By hypothesis
$$
\Omega=\{(\al_1,\al_2),\dots,(\al_1,\al_{m-1})\}\cup
\{(\al_2,\al_m),\dots,(\al_{m-1},\al_{m})\}\cup\Omega'
$$
where
$$
\Omega'=\{(\al_i,\al_j)\,|\,
1<i<j<m\}.
$$
Let
$$
K=I\cap (M((\rightarrow,\be_1))M((\al_1,\leftarrow))),
$$
$$
A=I\cap M([\al_1,\be_1]),
$$
$$
B=I\cap M((\al_1,\be_1)),
$$
$$
L=M_{\al_{1}\al_{2}}\cdots
M_{\al_{1}\al_{m-1}}M_{\al_{1}\be_{1}}M_{\al_{2}\al_{m}}\cdots
M_{\al_{m-1}\al_{m}}.
$$
Note that $A$ (resp. $B$) is the McLain group associated to
$\{\al_1,\dots,\al_{m}\}$ (resp. $\{\al_2,\dots,\al_{m-1}\}$). In
particular, $B\cong U_{m-2}(R)$.  Moreover, $B$ and $L$ are respectively
$M((\al_1,\be_1))$ and $M^{\al_1\be_1}$ for the McLain group associated to
$\{\al_1,\dots,\al_{m}\}$. Furthermore, we have
$$
A=L\rtimes B\text{ and }I=K\rtimes A,
$$
so that for
$$
T=K\rtimes L,
$$
we have
$$
I=T\rtimes B.
$$
Our description of $L$ and the results of \S\ref{elemeM} ensure
that there is one and only one irreducible $L$-module $X$, up to
isomorphism, lying over $\lm_1$ (viewed as a linear character of
$M_{\al_1,\be_1}$ via (\ref{rma})). Here $\dm(X)=|R|^{m-2}$ and
$(X,X)_L=1$.

We deduce from \S\ref{elemeM} and the above interpretation of $A$,
$B$ and $L$ that the action of $L$ on $X$ can be extended to $A$.
Let $V$ be the vector space $X$ viewed as an $A$-module under this
action and let $S:A\to\GL(V)$ be the associated representation.

On the other hand, since
$\Omega\cap(\rightarrow,\be_1)=\emptyset$, we see that $K$ is a
subgroup of $H$. In particular, $\lm$ is defined on $K$ (a
critical and subtle point). As $I=K\rtimes A$, we may define
$P:I\to\GL(V)$ by
$$
P(ak)=S(a)\lm(k).
$$
By Theorem \ref{corin}, $I$ stabilizes $\lm$, so $P$ is a group
homomorphism and $V$ is an irreducible $I$-module. Since $(X,X)_L=1$, we have
$(V,V)_I=1$. Moreover, by construction, $V$
lies over $W$ and
$$
(W,\res_H^I V)_H=|R|^{m-2}=\dm(V).
$$

We claim that an irreducible $I$-module $V'$ lies over $W$ if and
only if $V'\cong U\otimes V$, where $U$ is an irreducible
$B$-module viewed as $I$-module (and hence acted upon trivially by
$T$).

Indeed, let $U$ be an irreducible $B$-module viewed as $I$-module. Here
$\res^I_T V$ is irreducible (since so is $\res^I_L=X$) and, by definition, $U$ is acted upon trivially by~$L$, so, by Theorem \ref{cli4}, $U\otimes V$ is an
irreducible $I$-module. It actually lies over $W$. In fact, since
$T$ acts trivially on $U$ we see that
$$
(W,\res_H^I U\otimes V)_H=\dm(U)\dm(V).
$$

By Theorem \ref{cli4} the $U\otimes V$ are non-isomorphic and
since $F$ is a splitting field for $U_{m-2}(R)$, we have
$(U,U)_B=1$, so by Lemma
\ref{forma}
\begin{equation}
\label{cagas}
(U\otimes V,U\otimes V)_I=1.
\end{equation}
On the other hand, Theorem \ref{inv} gives
$$
(V(D,f),V(D,f))_M=[I:H].
$$
Let
$$
Z=\underset{\chi\in\Irr(B)}\bigoplus
\dm(U_\chi)\dm(V)\;\ind_I^M (U_\chi\otimes V).
$$
By Theorem \ref{fund2} and (\ref{cagas}), for each $\chi\in \Irr(B)$, we have
$$
(\ind_I^M (U_\chi\otimes V),\ind_I^M (U_\chi\otimes V))_M=1,
$$
so
$$
(Z,Z)_M=\dm(V)^2\underset{\chi\in\Irr(B)}\sum\dm(U_\chi)^2=\dm(V)^2|B|=|R|^{2(m-2)}|B|=[I:H].
$$
But $V(D,f)$ is completely reducible by Theorems \ref{fund2} and
\ref{xasc}, so $V(D,f)=Z$.
\end{proof}

\begin{cor}\label{mux} Suppose that $R$ is finite, $D$ has a special
triple, and $\Lambda$ is well-ordered by $\geq$. Then $V(D,f)$ has
a repeated irreducible constituent.
\end{cor}

\begin{proof} Immediate consequence of Example \ref{yeka2}.
\end{proof}

\begin{theorem}\label{mulfe} Suppose $R$ is finite and $F$ is a splitting field for $I$ over $\lm$.
Assume, in addition, $\Lambda$ is well-ordered by $\geq$. Then
$V(D,f)$ is multiplicity free if and only if $I/H$ is abelian
(this is equivalent to $[\Omega,\Omega]\subseteq \Gamma$) and $D$
has no special triple.
\end{theorem}

\begin{proof} If $I/H$ is
abelian and $D$ has no special triple then $V(D,f)$ is
multiplicity free by Theorems \ref{tx} and \ref{xasc}. If $D$ has
a special triple then Corollary \ref{mux} shows that $V(D,f)$ has
multiplicity, while if $D$ has no special triple but $I/H$ is
non-abelian then Theorems \ref{tx} and \ref{xasc} show that
$V(D,f)$ has multiplicity.
\end{proof}

\begin{exa}\label{ulex} Suppose $R$ is finite, $\Lambda$ is well-ordered by $\geq$, and $D$ satisfies
$$
\al_1<\al_2<\al_3=\be_1<\al_4=\be_2<\cdots<\al_{m}=\be_{m-2}<\be_{m-1}<\be_{m},
$$
where $m\geq 3$. Let $n$ be the largest integer such that $2n\leq
m$. Then

(a) If $m=2n+1$ there is one and only one irreducible $I$-module,
say $V$, up to isomorphism, lying over $\lm$. Moreover,
$\dm(V)=|R|^n$ with $\ind_I^M V$ is irreducible~and
$$
V(D,f)\cong |R|^n\ind_I^M V.
$$

(b) If $m=2n$ there is are exactly $|R|$ irreducible $I$-modules,
say $V_1,\dots,V_{|R|}$, up to isomorphism, lying over $\lm$.
Moreover, each $\dm(V_i)=|R|^{n-1}$, the $M$-modules $\ind_I^M
V_1,\dots,\ind_I^M V_{|R|}$ are irreducible and non-isomorphic,
and
$$
V(D,f)\cong |R|^{n-1}\ind_I^M V_1\oplus\cdots\oplus
|R|^{n-1}\ind_I^M V_{|R|}.
$$
\end{exa}

\begin{proof} By hypothesis
$$
\Omega=\{(\al_1,\al_2),(\al_2,\al_3)\dots,(\al_{m-1},\al_{m})\}.
$$
Set
$$
\Omega_0=\{(\al_1,\al_2),(\al_3,\al_4),\dots,(\al_{2n-1},\al_{2n})\},
$$
which is a closed abelian subset of $\Phi$. Let
$$
S=\underset{(\ga,\de)\in\Gamma_0}\prod M_{\ga\de}\;
(\text{internal direct product inside }I)
$$
and
$$
I_0=M(\Gamma\cup \Omega_0)=H\rtimes S.
$$
By Theorem \ref{corin}(c), $\lm$ is $I$-invariant. Thus, given any
$\mu\in\Hom(S,F^*)$, the map $\lm_\mu:M(\Gamma_0)\to F^*$, given
by
$$
\lm_\mu(hs)=\lm(h)\mu(s),\quad h\in H, s\in S,
$$
is a group homomorphism extending $\lm$. We readily verify that
$I_0\unlhd I$ and the inertia group of $\lm_\mu$ in $I$ is $I_0$.
Let $W_\mu$ be the vector space $W$ acted upon $I_0$ via
$\lm_\mu$. Then $V_\mu=\ind_{I_0}^I W_0$ is irreducible and
$(V_\mu,V_\mu)_I=1$ by Theorem \ref{cl2}. Moreover, it is obvious
that $V_\mu$ lies over $W$.

Suppose first $m=2n+1$. Then
$$
\dm(V_\mu)=[I:I_0]=|R|^{n}=[I_0:H].
$$
Thus Theorem \ref{fund2}(a) implies $\ind_H^I W\cong |R|^n V_\mu$.
This proves that the $V_\mu$ is independent of $\mu$, up to
isomorphism. Now apply Theorems \ref{fund2}(b) and \ref{xasc}.

Suppose next $m=2n$. Then
$$
\dm(V_\mu)=[I:I_0]=|R|^{n-1}\text{ and }[I_0:H]=|R|^n.
$$
Each $\lm_\mu$ is conjugate under $I$ to exactly
$[I:I_0]=|R|^{n-1}$ extensions of $\lm$, whereas by Lemma
\ref{cara2},  $\lm$ has a total of $|S|=|R|^n$ extensions to
$I_0$. Thus, these $|R|^n$ extensions fall into exactly $|R|$
classes upon conjugation by $I$. On the other hand, by Clifford's
theory, $V_\mu\cong V_\nu$ if and only if $\mu$ and $\nu$ are in
the same $I$-conjugacy class. Let $\mu_1,\dots,\mu_{|R|}$ be
representatives for the  $I$-conjugacy classes of extensions of
$\lm$ to $I_0$. Since
$$
\dm(\ind_H^I W)=[I:H]=|R|\times |R|^{n-1}\times |R|^{n-1},
$$
it follows from Theorem \ref{fund2}(a) that
$$
\ind_H^I W\cong |R|^{n-1}V_{\mu_1}\oplus\cdots\oplus
|R|^{n-1}V_{\mu_{|R|}}.
$$
Now apply Theorems \ref{fund2}(b) and \ref{xasc}.
\end{proof}

\begin{exa}\label{coco} Suppose $R$ is finite and $D=D_1\cup\cdots\cup D_\ell$ is a disconnection,
where each $D_i$ has size $m_i=2n_i+1$ and is as in Example
\ref{ulex}.  Assume, in addition, that $\Lambda$ is well-ordered
by $\geq$. Let $f_i$ be the restriction of $f$ to $D_i$, $1\leq
i\leq \ell$. Then
$$V(D,f)\cong |R|^n (Y_1\otimes\cdots\otimes Y_\ell),$$ where $n=n_1+\cdots+n_\ell$,
each $Y_i$ is the irreducible $M$-module from Example \ref{ulex}
satisfying $V(D_i,f_i)\cong |R|^{n_i}Y_i$, and
$Y_1\otimes\cdots\otimes Y_\ell$ is an irreducible $M$-module.
\end{exa}

\begin{proof} This follows from Theorem \ref{cli6} and Example \ref{ulex}.
\end{proof}

\noindent{\bf Acknowledgment.} We thank D. Stanley for useful conversations.

\end{document}